\pdfoutput=1
\documentclass[10pt,a4paper]{article}
\usepackage[left=1in, right=1in, top=1in,  bottom=1in]{geometry}

\usepackage[utf8]{inputenc}
\usepackage[english]{babel}
\usepackage{amsmath,amsfonts,amssymb, amsthm}
\usepackage{comment}

\makeatletter
\DeclareRobustCommand
  \myvdots{\vbox{\baselineskip4\p@ \lineskiplimit\z@
    \hbox{.}\hbox{.}\hbox{.}}}
\makeatother

\usepackage{enumerate}
 
 \usepackage[shortlabels]{enumitem}

\usepackage{tikz}
\usetikzlibrary{calc,fit, patterns,backgrounds}
\tikzset{rededge/.style={densely dotted, ultra thick, color=red},
        blueedge/.style={very thick, color=blue},
        vertex/.style={circle, inner sep=1pt, fill=black!100, draw}}
\pgfsetlinewidth{1pt}

\usepackage[labelfont=bf]{caption}
\usepackage{subcaption}
\usepackage{hyperref}
\usepackage{cleveref}
\hypersetup{pdfborder={0 0 0.5 [2 3]}}
\hypersetup{citebordercolor={1 0 1}}

\newcounter{case}
\numberwithin{figure}{section}
\numberwithin{equation}{section}

\newtheorem{thm}{Theorem}
\newtheorem{cor}[thm]{Corollary}
\newtheorem{lmm}[thm]{Lemma}
\newtheorem{proposition}[thm]{Proposition}
\newtheorem{rmk}[thm]{Remark}

\newtheorem{dfn}{Definition}
\newtheorem{notation}[dfn]{Notation}
\newtheorem{clm}{Claim}[case]

\usepackage{stmaryrd}
\DeclareDocumentCommand \To{o} {
  \IfNoValueTF {#1} {
    \shortrightarrow%
  }{
    \mathrel{\raisebox{-3pt}{$\overset{#1}{\rightsquigarrow}$}}%
  }%
}

\newcommand{\Sleq}[1]{S_{\leq #1}}
\newcommand{\Sgt}[1]{S_{> #1}}
\newcommand{\Sgeq}[1]{S_{\geq #1}}
\newcommand{\Slt}[1]{S_{< #1}}

\DeclareMathOperator{\cp}{cp}

\begin{document}

\title{Partitioning 2-edge-coloured bipartite graphs into monochromatic cycles}

\author{Fabrício Siqueira Benevides\thanks{Departamento de Matemática, Universidade Federal do Ceará, Brazil (\url{fabricio@mat.ufc.br})}
\thanks{This work was funded by the Conselho Nacional de Desenvolvimento Científico e Tecnológico (CNPq), FAPESB (EDITAL FAPESB N. 012/2022 Universal N.APP0044/2023) and Coordenação de Aperfeiçoamento de Pessoal de Nível Superior (CAPES) – Brasil  (Finance Code 001).}
\and Arthur Lima Quintino\footnotemark[1]
\and \underline{Alexandre Talon}\thanks{Univ. Grenoble Alpes, CNRS, Grenoble INP, G-SCOP, Grenoble, France (\url{alexandre.talon@univ-grenoble.fr})}}

\maketitle

\begin{abstract}
Given an $r$-edge-colouring of the edges of a graph $G$, we say that it can be partitioned into $p$ monochromatic cycles when there exists a set of $p$ vertex-disjoint monochromatic cycles covering all the vertices of $G$. In the literature of this problem, an edge and a single vertex both count as a cycle.
    
We show that for every $2$-colouring of the edges of a complete balanced bipartite graph, $K_{n,n}$, it can be partitioned into at most 4 monochromatic cycles. This type of question was first studied in 1970 for complete graphs and in 1983, by Gyárfás and Lehel, for $K_{n,n}$. In 2014, Pokrovskiy showed for all $n$ that, given any $2$-colouring of its edges, $K_{n,n}$ can be partitioned into at most three monochromatic paths. It turns out that finding monochromatic cycles instead of paths is a natural question that has also been raised for other graphs. In 2015, Schaudt and Stein showed that 14 cycles are sufficient for sufficiently large $2$-edge-coloured $K_{n,n}$.
\end{abstract}

\section{Introduction}\label{sec-intro}

Throughout this paper, an $r$-colouring of a graph $G$ is a colouring of its \textbf{edges} with up to $r$ different colours. For such an $r$-colouring we say that \emph{the coloured graph is partitioned into $p$ monochromatic cycles} when there exist a set of $p$ vertex-disjoint monochromatic cycles covering all the vertices of~$G$. We are interested in the \emph{cycle partitioning number of $G$}, that is the smallest number $\cp_r(G)$ such that, for every $r$-colouring of $G$, it can be partitioned into at most $\cp_r(G)$ monochromatic cycles. In the literature of this problem, an edge and a single vertex both count as a cycle.\\

In the 1970's, Lehel made an influential conjecture stating that when the edges of a complete graph $K_n$ are coloured with two colours, it can be partitioned into at most two monochromatic cycles of different colours (but note that the definition of $\cp_r(G)$ does not require the cycles to have distinct colours). {\L}uczak, Rödl and Szemerédi~\cite{lczk-rdl-szmrd98} confirmed Lehel’s conjecture for sufficiently large complete graphs, after some preliminary work of Gyárfás~\cite{grfs83}. This was later improved by Allen~\cite{alln08}, who also gave a proof that there exists some $n_0$ such that Lehel’s conjecture is true for $n\ge n_0$, but avoiding the use of the regularity method and, therefore, providing a smaller value for $n_0$. Allen’s proof also gives a polynomial time algorithm to find the two cycles. In 2010, Bessy and Thomassé~\cite{bss-tmss10} proved it for every complete graph, with a surprisingly short proof. To our knowledge, for a fixed $r$ the best upper bound is $\cp_r(K_n) \le 100r \log r$, for $n$ large enough (see~\cite{grfs-rsznk-skz-szmrd06.1}), while it was conjectured by Erd{\H{o}}s, Gy{\'a}rf{\'a}s and Pyber \cite{erds-grfs-pbr91} that $\cp_r(K_n) \le r$. However, Pokrovskiy \cite{pkvsk14} has shown that this conjecture is false for $r\ge 3$, constructing infinitely many $r$-coloured complete graphs that cannot be vertex-partitioned into $r$ monochromatic cycles, what motivated researchers on working on weaker conjectures about partitioning all but a few number of vertices. In the same article~\cite{pkvsk14}, he proves that for every $3$-colouring of $K_n$, there exists a partition (of all vertices) into at most $3$ monochromatic \emph{paths}, improving a result from~\cite{grfs-rsznk-skz-szmrd11}. 

In the 2-colour case, if we replace cycles by paths, the problem becomes much easier. As noted by Gerencsér and Gyárfás \cite{grcsr-grfs67} in 1967: \emph{the vertex set of any 2-coloured complete graph can be partitioned into two paths of different colours}. To prove it, take any red-blue colouring of $K_n$ and suppose that $R$ and $B$ are vertex disjoint red and blue paths, with respective endpoints $r$ and $b$. Let $v \not\in V(R)\cup V(B)$. If $vr$ is red or $vb$ is blue then we can extend $R$ or $B$ accordingly. Otherwise, without loss of generality $rb$ is red and $\{R\cup rb \cup bv, B \setminus b\}$ is a pair of red-blue paths covering one extra vertex. Together with the previous paragraph, this shows that although the problems for cycles and paths may look similar, they have very different difficulties and outcomes.

The one with cycles has also been studied. Denoting by $\alpha(G)$ the independence number of $G$, Sárközy \cite{sarkozy2011monochromatic} conjectured that $\cp_r(G) \leq r\alpha(G)$ for every $r$-colouring of $G$, while in \cite{balogh2014partitioning} it was proved that for fixed $r$ and $\epsilon>0$, for sufficiently large graphs, one can partition at least $(1-\epsilon)|V(G)|$ of the vertices of an $r$-coloured graph $G$ into at most $r\alpha(G)$ monochromatic cycles. In the same article, motivated by a question raised by Schelp, they also prove that if $r=2$ and $G$ has minimum degree $(1+\epsilon)3|V(G)|/4$, then one can partition a set of at least $(1-\epsilon)|V(G)|$ vertices of $G$ with two monochromatic cycles. This was improved in \cite{debiasio2017monochromatic}: Debasio and Nelsen showed how to absorb the remaining $\epsilon|V(G)|$ vertices. Furthermore, Letzter \cite{letzter2019monochromatic} removed the $\epsilon$ from the bound for the minimum degree. Putting everything together, a $2$-coloured graph with minimum degree at least $3|V(G)|/4$ can be partitioned into two monochromatic paths, proving the conjecture from \cite{balogh2014partitioning}. We refer the reader to a 2016 survey by Gyásfás \cite{grfs16} for other results up to that year. Some results for hypergraphs appear in \cite{gyarfas2013monochromatic} and \cite{stein2023monochromatic} (the later only for paths); and for random graphs in \cite{korandi2018monochromatic}.\\

A very natural extension to the problem is to work on complete multipartite graphs. In particular when $G$ is the balanced complete bipartite graph $K_{n,n}$. In 1983, Gyárfás~\cite{grfs83} stated that every balanced complete bipartite graph with a 2-colouring that is not a split colouring (see \Cref{dfn-splitcolouring}, below) contains two vertex-disjoint monochromatic \emph{paths} that partition all but at most one of its vertices (although the proof actually also appears implicitly in a paper of Gyárfás and Lehel~\cite{grfs-lhl73}, in 1973). This result was later extended by Pokrovskiy (see~\Cref{thm-Knn-simplepath}, below): he managed to cover all vertices.

The first general upper bounds for $\cp_r(K_{n,n})$ were given by Haxell \cite{haxell1997partitioning} and by Peng, Rödl, and Ruci\'nski \cite{png-rdl-rcnsk02}. The current best known result is that $\cp_r(K_{n,n}) \le 4r^2$ if $n$ is large enough (see \cite{lng-stn15}). This improved a result in \cite{cnln-stn16}. For the lower bound, an easy construction shows that we need at least $2r-1$ cycles \cite{pkvsk14,lng-scht-stn15} for a $r$-coloured $K_{r,r}$. Start with one partition class: blow up each vertex in this class to a set of size $r$. In the other partition class blow up only one vertex to a set of size $r(r-1) + 1$. In~\cite{lng-scht-stn15}, Lang, Schaudt and Stein say they believe that $2r-1$ is the correct answer and claim that this was confirmed for $r=2$ for large values of $n$ (by another author who informed them through personal communication). However, this was 9 years ago and a proof was never published (and only holds for large $n$). To our knowledge the best published bound for $r=2$ was given by Schaudt and Stein~\cite{schaudt2019partitioning}: they showed that $\cp_2(K_{n,n}) \le 12$ for sufficiently large values of $n$. In \cite{lng-scht-stn15} it was also proved that when $r=3$ five monochromatic cycles are enough to partition \emph{almost all} vertices of $K_{n,n}$. Finally, \cite{schaudt2019partitioning} also considered the problem for complete $k$-partite graphs $G$ under the optimal condition that $G$ contains at most $|V(G)|/2$ vertices in every partition class. They show that for $r=2$, the vertices of such a $2$-coloured and sufficiently large $G$ can be partitioned into at most two monochromatic paths or at most 14 monochromatic cycles.\\

Here, we show that $\cp_2(K_{n,n}) \le 4$ \textbf{for every} natural number $n$ (see~\Cref{thm-Knn-4cycles}). Therefore $\cp_2(K_{n,n})\in \{3, 4\}$. Before proving it, we shall define the basic concepts that we will use throughout this paper.

\begin{dfn}
We say that a path is \textbf{simple} when it is the union of a blue path, $v_1v_2\ldots v_i$ and a red path, $v_iv_{i+1}\ldots v_k$. We call $v_i$ the \textbf{turning point} of the path. 
\end{dfn}

\begin{dfn}
We say that a bipartite graph is \textbf{balanced} when its partition classes have the same number of vertices. 
\end{dfn}

\begin{dfn}\label{dfn-splitcolouring}
Let $G$ be a bipartite graph with partition classes $X$ and $Y$, whose edges are coloured red and blue. The colouring on $G$ is \textbf{split} when $X$ and $Y$ can each be partitioned into two non-empty sets, $X=X_1\cup X_2$ and $Y=Y_1\cup Y_2$, such that all edges between $X_i$ and $Y_j$ are red for $i=j$ and blue for $i\neq j$ (see Figure~\ref{fig-splitcolouring}).
\end{dfn}

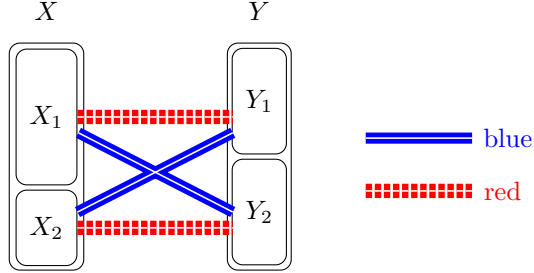
\begin{figure}[!ht]
\centering
\begin{tikzpicture}[scale=1.4]

\tikzset{set1/.style={inner sep=5pt, draw, rounded corners, minimum height=1cm}}
\tikzset{set2/.style={inner sep=5pt, draw, rounded corners, minimum height=1.8cm}}
\tikzset{set3/.style={inner sep=5pt, draw, rounded corners, minimum height=1.4cm}}

\node (x) at (0.5,2.5) {$X$};
\node [set2] (x1) at (0.5,1.5) {$X_1$};
\node [set1] (x2) at (0.5,0.45) {$X_2$};

\node (y) at (2.5,2.5) {$Y$};
\node [set3] (y1) at (2.5,1.65) {$Y_1$};
\node [set3] (y2) at (2.5,0.6) {$Y_2$};

\draw [rounded corners] (0.15,0.05) rectangle (0.85,2.2);
\draw [rounded corners] (2.2,0.05) rectangle (2.8,2.2);

\coordinate (y1red) at ([shift={(-0.25,-0.15)}]y1);
\coordinate (y2red) at ([shift={(-0.25,-0.15)}]y2);
\coordinate (x1blue) at ([shift={(0.3,-0.15)}]x1);
\coordinate (y1blue) at ([shift={(-0.25,-0.3)}]y1);
\coordinate (y2blue) at ([shift={(-0.25,0)}]y2);
\draw [densely dotted, double, line width=2.4pt, color=red] (x1)--(y1red) (x2)--(y2red);
\draw [double, line width=1.8pt, color=blue] (x1blue)--(y2blue) (x2)--(y1blue);

\draw[densely dotted, double, line width=2.4pt, color=red] (3.5, 0.8) -- (4.5,0.8) node[right] {red};
\draw[double, line width=1.8pt, color=blue] (3.5, 1.3) -- (4.5,1.3) node[right] {blue};

\end{tikzpicture}
\caption{A split colouring.}
\label{fig-splitcolouring}
\end{figure}
In 1983, Gyárfás and Lehel~\cite{grfs83, grfs-lhl73} showed that every balanced complete bipartite graph with a 2-colouring that is not split contains two vertex-disjoint monochromatic paths that cover all but at most one of its vertices. This result was later extended by Pokrovskiy.

\begin{thm}[Pokrovskiy \cite{pkvsk14}, 2014]\label{thm-Knn-simplepath}
Let $G$ be a balanced complete bipartite graph whose edges are coloured red and blue. There is a vertex-partition of $G$ into two monochromatic \textbf{paths with different colours} if, and only if, the colouring on $G$ is not split.
\end{thm}

\begin{rmk}\label{cor-slit-no-2-cycles-same-color}
 Let $G$ be a complete bipartite graph with a split colouring of its edges into red and blue. The vertices of $G$ cannot be partitioned into two monochromatic cycles of different colours.
\end{rmk}
\begin{proof} 
By \Cref{thm-Knn-simplepath} (or by a very simple argument), the graph $G$ cannot be partitioned into two paths of different colours. Therefore, it cannot be partitioned into two cycles of different colours either.
\end{proof}

However, for split colourings the following is (trivially) true. We nevertheless include a proof for completeness.

\begin{proposition}\label{lmm-Kn-nsplit-3cycles}
Let $G$ be a balanced complete bipartite graph whose edges are coloured red and blue. If the colouring on $G$ is split, then $G$ can be (vertex-)partitioned into at most $3$ monochromatic cycles. Furthermore, two cycles suffice if and only if $|X_1| = |Y_1|$ or $|X_1|=|Y_2|$.
\end{proposition}
\begin{proof}
Let $X$ and $Y$ be the partition classes of $G$. And $X = X_1\cup X_2$, $Y = Y_1\cup Y_2$ as in the definition of split colouring (\Cref{dfn-splitcolouring}). We may assume without loss of generality that $|X_1|\geq |Y_1|$. As $|X|=|Y|$, we have $|X_1|-|Y_1|=|Y_2|-|X_2|$ and, therefore, $|X_2|\leq |Y_2|$. Hence, $G$ can be vertex-partitioned into a red cycle that covers all vertices in $Y_1$ and $|Y_1|$ vertices in $X_1$, a red cycle that covers all vertices in $X_2$ and $|X_2|$ vertices in $Y_2$ and a blue cycle that covers all the remaining vertices in $X_1$ and $Y_2$. And clearly, if $|X_1| = |Y_1|$, then $|X_2| = |Y_2|$ and two cycles (of the same colour) suffice.

On the contrary, suppose that two monochromatic cycles suffice. By \Cref{cor-slit-no-2-cycles-same-color} they must be of the same colour. Suppose, without loss of generality that both are red. Each cycle must be entirely contained in $X_1\cup Y_1$ or $X_2\cup Y_2$. And because all the sets $X_1$, $X_2$, $Y_1$ and $Y_2$ are non-empty, one cycle must cover vertices in $X_1\cup Y_1$ and the other in $X_2\cup Y_2$. And each cycle covers as many vertices in $X$ as in $Y$. Therefore $|X_1| = |Y_1|$ and $|X_2| = |Y_2|$. Note: similarly, blue cycles suffice if and only if $|X_1| = |Y_2|$ and $|X_2| = |Y_1|$.
\end{proof}

We could naturally wonder whether, given any non-split 2-colouring, a balanced bipartite graph can be partitioned into at most 2 monochromatic cycles. In \Cref{prop:ex-almost-best} we give an example of a colouring that is not split and needs at least 3 cycles. 

For upper bounds, using the regularity method, Schaudt and Stein~\cite{schaudt2019partitioning} showed that for every $k$, a 2-coloured sufficiently large $k$-partite graph $G$, such that no partition class of $G$ contains more than half of its vertices, can be partitioned into at most $14$ monochromatic cycles. For $k=2$, they proved that this can be done with at most $12$ monochromatic cycles.

Our main contribution is to reduce this number to only $4$ monochromatic cycles, for all balanced bipartite graphs. Note that our result does not require $G$ to be large.

\begin{thm}\label{thm-Knn-4cycles}
If $G$ is a balanced complete bipartite graph whose edges are coloured red and blue, then $G$ can be partitioned into at most $4$ monochromatic cycles.
\end{thm}

\section{Proving \texorpdfstring{\Cref{thm-Knn-4cycles}}{main theorem}}
\label{sub-Kn,n-4cycles}

Thanks to \Cref{lmm-Kn-nsplit-3cycles}, we only have to prove \Cref{thm-Knn-4cycles} when the given colouring is not split. Our proof is self-contained apart from the following result.

\begin{cor}[Stein \cite{stein2023monochromatic}, 2022]\label{cor-Knn-pathcycle}
Let $G$ be a balanced complete bipartite graph whose edges are coloured red and blue. If the colouring on $G$ is not split, then $G$ can be vertex-partitioned into a monochromatic path and a monochromatic cycle with different colours.
\end{cor}

We observe that the above is a corollary of \Cref{thm-Knn-simplepath}, and follows from a very short case analysis on the parity of the paths given by \Cref{thm-Knn-simplepath}. This is done in details in Section 4.1 of~\cite{stein2023monochromatic}.
We also note that Section~4.2 of~\cite{stein2023monochromatic} presents an alternative proof of \Cref{thm-Knn-simplepath} that is also short and self-contained.

In view of \Cref{cor-Knn-pathcycle}, it is enough to prove the following lemma. 

\begin{lmm}\label{lmm-Knnwithpath-3cycles}
Let $G$ be a balanced complete bipartite graph whose edges are coloured red and blue. If $G$ has a monochromatic Hamiltonian path, then $G$ can be vertex-partitioned into at most $3$ monochromatic cycles.
\end{lmm}

In fact, let $H$ be any balanced complete bipartite graph with a 2-colouring of its edges. When we apply \Cref{cor-Knn-pathcycle} to $H$, the monochromatic cycle we obtain must be even because $H$ is bipartite. And since $H$ is balanced, the monochromatic path given by \Cref{cor-Knn-pathcycle} is also even. So its vertices induce a balanced complete bipartite subgraph of $H$ that contains a monochromatic Hamiltonian path.

\Cref{lmm-Knnwithpath-3cycles} is best possible, in the sense that even with the extra condition that the colouring of the bipartite graph contains a Hamiltonian monochromatic path, there are 2-colourings that cannot be partitioned into 2 monochromatic cycles, as illustrated below.

\begin{proposition}\label{prop:ex-almost-best}
Let $G$ be a balanced complete bipartite graph on $2n$ vertices, with bipartition $S = \{x_1, \ldots, x_n\}\cup\{y_1,\ldots,y_n\}$. Take the colouring where:
\begin{itemize}
\item $x_1y_1$, $x_2y_2$ and all edges of the form $x_1y_i$ and $y_1x_i$ for $i \geq 3$ are blue;
\item the edges $x_2y_i$ for $i \geq 4$ are blue;
\item all the other edges are red.
\end{itemize}
Then $G$ contains a Hamiltonian red path and cannot be partitioned into two monochromatic cycles.
\end{proposition}

\begin{proof}
Since the graph induced by $\Sgeq{3} = \{x_3, \ldots, x_n\}\cup\{y_3,\ldots,y_n\}$ has all edges coloured red, there is a path $P$ starting in $x_3$ and ending in $y_3$ and passing through all vertices of $\Sgeq{3}$. The red path $x_1y_2x_3Py_3x_2y_1$ is a red Hamiltonian path of $G$.

Now we show that $G$ cannot be partitioned into 2 monochromatic cycles. First, the vertex $y_1$ cannot belong to a cycle with more than two vertices: such a cycle would be blue because only one red edge is incident to $y_1$, and it would then contain one $x_i$ with $i \geq 2$ but such an $x_i$ only has one incident blue edge, towards $y_1$, hence it could not continue further. Note that since the graph is balanced and bipartite, we can assume that no cycles contain a single vertex. Hence $y_1$ does not belong to any cycle.
Let us now assume we have a partition into two monochromatic cycles. One cycle must contain $y_1$ and exactly one of the $x_l$'s. The other cycle must contain $y_2$, therefore it cannot be blue since only one blue edge is incident to that vertex. So it needs to be a red cycle, and then cannot contain $x_1$ (only one incident red edge), hence the first cycle must be $x_1y_1$. Then this second cycle, which is red, must contain both $x_2$ and $y_2$ but $x_2$ has only one incident red edge in $G \setminus \{x_1, y_1\}$ so this is a contradiction.
\end{proof}


Next we present the structure of the proof of \Cref{thm-Knn-4cycles}. In all that follows, we consider a balanced complete bipartite graph $G$ whose edges are coloured red and blue \textbf{and which contains a red Hamiltonian path}. We also consider that this colouring is not split, and we want to show that there exists a partition of $G$ into at most 3 monochromatic cycles.

Our idea is to search for more structure in the above graph $G$. Let us say that each partition class of $G$ has $n$ vertices. We will label its vertices in a particular way. Then, in order to prove \Cref{lmm-Knnwithpath-3cycles}, we proceed by induction on a parameter $k$, $1\le k \le n$, with a well chosen hypothesis. We will try to show the existence of a specific subgraph which we will define later as \emph{a blue even plait}: the first $k/2$ vertices with even indices from both classes (of the bipartition) inducing a complete blue bipartite graph, and the same with the first $k/2$ vertices with odd indices. This exists for $k=1$, and we will show that if $G$ cannot be partitioned into three monochromatic cycles and the induction hypothesis holds for $k$, then it also holds for $k+1$. Iterating this will give us our result, because a blue even plait shall be trivially partitioned into 2 blue cycles. The proof will mostly consists in looking at specific edges and proving that they are blue (though we also need to argue that a few edges are red), with the goal to increase the order of the blue even plait.\\

We split this section into two subsections. Subsection~\ref{subsec:preliminary} contains some preliminary results that we will use many times and help us to organise the other subsection in a concise way. Subsection~\ref{subsec:mainproof} is devoted to the main proof of \Cref{lmm-Knnwithpath-3cycles}, which is longer. The preliminary results include a lemma with a short inductive proof which suffices to deduce a weakening of \Cref{lmm-Knnwithpath-3cycles} with 4 cycles, which in turn (by the above arguments) implies a weakening of \Cref{thm-Knn-4cycles} with 5 cycles.

\subsection{Preliminary structural results}
\label{subsec:preliminary}
Throughout the rest of this paper we deal with balanced complete bipartite graphs whose edges are coloured red and blue and which have a monochromatic Hamiltonian path. Therefore, in order to avoid exhaustive repetitions, it will be convenient to establish the following convention for the order of the vertices in such a path.

\begin{dfn}\label{dfn-zigzaggraph}
A balanced complete bipartite graph with $2n$ vertices whose edges are coloured red and blue is a \textbf{red zigzag graph} when its partition classes are $X=\lbrace x_1,\ldots,x_n \rbrace$ and $Y=\lbrace y_1,\ldots,y_n \rbrace$ and the Hamiltonian path \[
  P=(x_1,y_2,x_3,y_1,\ldots,x_4,y_3,x_2,y_1)
\] is red (see Figure~\ref{fig-zigzaggraph}).

For every $k\leq n$, we also define the sets:
\begin{align*}
 \Sleq{k} &= \{x_i : i\leq k\}\cup\{y_i : i\leq k\}, \\
 \Slt{k} &= \{x_i : i< k\}\cup\{y_i : i < k\}, \\
 \Sgeq{k} &= \{x_i : i\geq k\}\cup\{y_i : i\geq k\}, \\
 \Sgt{k} &= \{x_i : i> k\}\cup\{y_i : i> k\}.
\end{align*}
We denote by $x_iPy_j$ the red Hamiltonian subpath of $P$ starting from vertex $x_i$ and ending at $y_j$, both end-vertices being included.
\end{dfn}

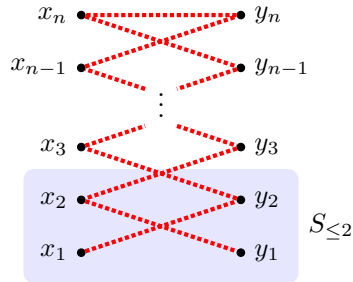
\begin{figure}[!ht]
\centering
\begin{tikzpicture}[scale=0.7]
\node [vertex, label={[label distance=0cm]180:$x_1$}] (x1) at (0,0) {};
\node [vertex, label={[label distance=0cm]180:$x_2$}] (x2) at (0,1) {};
\node [vertex, label={[label distance=0cm]180:$x_3$}] (x3) at (0,2) {};
\node [vertex, label={[label distance=0cm]180:$x_{n-1}$}] (xn-1) at (0,3.5) {};
\node [vertex, label={[label distance=0cm]180:$x_n$}] (xn) at (0,4.5) {};

\node (dots) at (1.5,2.8) {$\myvdots$};

\node [vertex, label={[label distance=0cm]0:$y_1$}] (y1) at (3,0) {};
\node [vertex, label={[label distance=0cm]0:$y_2$}] (y2) at (3,1) {};
\node [vertex, label={[label distance=0cm]0:$y_3$}] (y3) at (3,2) {};
\node [vertex, label={[label distance=0cm]0:$y_{n-1}$}] (yn-1) at (3,3.5) {};
\node [vertex, label={[label distance=0cm]0:$y_n$}] (yn) at (3,4.5) {};

\draw [rededge] (x1)--(y2)--(x3)--(1.2,2.4) (1.8,2.4)--(y3)--(x2)--(y1);
\draw [rededge] (1.2,3.1)--(xn-1)--(yn)--(xn)--(yn-1)--(1.8,3.1);

\begin{scope}[on background layer]
\node[rectangle, rounded corners, fill=blue!10, fit=(x1)(y2), inner xsep=20pt, inner ysep=10pt, label={right:$\Sleq{2}$}]{};
\end{scope}
\end{tikzpicture}
\caption{Labelling of the red Hamiltonian path $P=(x_1,y_2,x_3,y_4,\ldots,x_4,y_3,x_2,y_1)$ of a red zigzag graph with $2n$ vertices. And the set $\Sleq{2}$.}
\label{fig-zigzaggraph}
\end{figure}

In addition to these definitions and to other ones that will appear later, we shall also establish some simple facts that will help us in the proofs of the main results of this section. Such facts will be just stated as remarks, since they are straightforward or can be easily checked.

\begin{rmk}\label{rmk-xiyixiyi+2}
Let $G$ be a red zigzag graph with $2n$ vertices. Then, for $i\leq n$:
\begin{enumerate}[(a)]
\item if the edge $x_iy_i$ is red, then the subgraph $G[\Sgeq{i}]$ has a red Hamiltonian cycle, namely $(x_iPy_i)$ (see Figure~\ref{fig-xiyixiyi+2-a}).
      
\item if the edge $x_iy_{i+2}$ (resp. $x_{i+2}y_i$) is red, then the subgraph $G[\Sgeq{i}]$ can be vertex-partitioned into two red cycles: the red cycle $(x_i,y_{i+2}Px_{i+2},y_{i+1})$ (resp. $(x_{i+1},y_{i+2}Px_{i+2},y_i)$) and the edge-cycle $(x_{i+1},y_i)$ (resp. $(x_i,y_{i+1})$) (see Figure~\ref{fig-xiyixiyi+2-b}).
\end{enumerate}
\end{rmk}

\begin{figure}[!ht]
\centering
\begin{subfigure}[!ht]{0.5\textwidth}
\centering
\begin{tikzpicture}[scale=0.8]
\node [vertex, label={[label distance=0cm]180:$x_i$}] (x1) at (0,0) {};
\node [vertex, label={[label distance=0cm]180:$x_{i+1}$}] (x2) at (0,1) {};
\node [vertex, label={[label distance=0cm]180:$x_{i+2}$}] (x3) at (0,2) {};
\node [vertex, label={[label distance=0cm]180:$x_{n-1}$}] (xn-1) at (0,3.5) {};
\node [vertex, label={[label distance=0cm]180:$x_n$}] (xn) at (0,4.5) {};

\node (dots) at (1.5,2.8) {$\myvdots$};

\node [vertex, label={[label distance=0cm]0:$y_i$}] (y1) at (3,0) {};
\node [vertex, label={[label distance=0cm]0:$y_{i+1}$}] (y2) at (3,1) {};
\node [vertex, label={[label distance=0cm]0:$y_{i+2}$}] (y3) at (3,2) {};
\node [vertex, label={[label distance=0cm]0:$y_{n-1}$}] (yn-1) at (3,3.5) {};
\node [vertex, label={[label distance=0cm]0:$y_n$}] (yn) at (3,4.5) {};

\draw [rededge] (1.8,2.4)--(y3)--(x2)--(y1)--(x1)--(y2)--(x3)--(1.2,2.4);
\draw [rededge] (1.2,3.1)--(xn-1)--(yn)--(xn)--(yn-1)--(1.8,3.1);

\end{tikzpicture}
\caption{}
\label{fig-xiyixiyi+2-a}
\end{subfigure}%
\begin{subfigure}[!ht]{0.5\textwidth}
\centering
\begin{tikzpicture}[scale=0.8]

\node [vertex, label={[label distance=0cm]180:$x_i$}] (x1) at (0,0) {};
\node [vertex, label={[label distance=0cm]180:$x_{i+1}$}] (x2) at (0,1) {};
\node [vertex, label={[label distance=0cm]180:$x_{i+2}$}] (x3) at (0,2) {};
\node [vertex, label={[label distance=0cm]180:$x_{n-1}$}] (xn-1) at (0,3.5) {};
\node [vertex, label={[label distance=0cm]180:$x_n$}] (xn) at (0,4.5) {};

\node (dots) at (1.5,2.8) {$\myvdots$};

\node [vertex, label={[label distance=0cm]0:$y_i$}] (y1) at (3,0) {};
\node [vertex, label={[label distance=0cm]0:$y_{i+1}$}] (y2) at (3,1) {};
\node [vertex, label={[label distance=0cm]0:$y_{i+2}$}] (y3) at (3,2) {};
\node [vertex, label={[label distance=0cm]0:$y_{n-1}$}] (yn-1) at (3,3.5) {};
\node [vertex, label={[label distance=0cm]0:$y_n$}] (yn) at (3,4.5) {};

\draw [rededge] (1.8,2.4)--(y3)--(x1)--(y2)--(x3)--(1.2,2.4);
\draw [rededge] (1.2,3.1)--(xn-1)--(yn)--(xn)--(yn-1)--(1.8,3.1);
\draw [rededge] (x2)--(y1);

\end{tikzpicture}
\caption{}
\label{fig-xiyixiyi+2-b}
\end{subfigure}

\caption{The subgraph $G[\Sgeq{i}]$ according to cases (a) and (b) of Remark \ref{rmk-xiyixiyi+2}.}
\label{fig-xiyixiyi+2}
\end{figure}
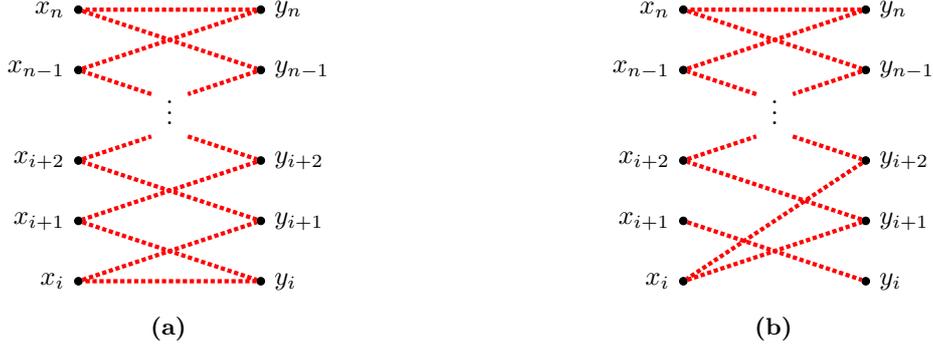

In the above language, proving Lemma \ref{lmm-Knnwithpath-3cycles} is the same as proving that \emph{every red zigzag graph can be vertex-partitioned into at most $3$ monochromatic cycles}. Therefore, the lemma below is a weaker version of Lemma \ref{lmm-Knnwithpath-3cycles} which allows one extra monochromatic cycle to be used. As compensation, this version gives us some additional information about the monochromatic cycles of the partition we obtained.\\

\begin{lmm}\label{lmm-Knn-withpath-4cycles}
\vspace*{-0.2cm}
If $G$ is a red zigzag graph with $2n$ vertices, then $G$ can be vertex-partitioned into at most $t$ monochromatic cycles satisfying one of the following conditions.
\begin{enumerate}[(i)]
\item $t=2$.
\item $t=3$ and the edge $x_1y_1$ is used in some blue cycle of the partition.
\item $t=4$ and the edges $x_1y_1$ and $x_2y_2$ are used in different blue cycles of the partition.
\end{enumerate}
\end{lmm}

\begin{proof}
We prove this lemma by induction on $n$. For $n\leq 2$, the graph $G$ satisfies condition (i) trivially. For $n\geq 3$, we may assume that $G$ does not satisfy condition (i), since otherwise we would be done. Hence, by Remark \ref{rmk-xiyixiyi+2}, we know that the edges $x_1y_1$, $x_1y_3$ and $x_3y_1$ are blue. Now, we apply the inductive hypothesis to the subgraph $G[\Sgt{1}]$. If $G[\Sgt{1}]$ satisfies conditions (i) or (ii), then we may simply take the edge $x_1y_1$ as a blue cycle and so we see that $G$ satisfies conditions (ii) or (iii), respectively. On the other hand, if $G[\Sgt{1}]$ satisfies condition (iii), then we may use the edges $x_1y_1$, $x_1y_3$ and $x_3y_1$ to extend the blue cycle of the partition of $G[\Sgt{1}]$ that uses the edge $x_3y_3$ and so we see that $G$ satisfies condition (iii) as well. In any case, the result follows.
\end{proof}

In view of the proof of Theorem \ref{thm-Knn-4cycles}, observe that the lemma above has a certain relevance by itself, since it implies that $5$ monochromatic cycles are sufficient to vertex-partition any balanced complete bipartite graph whose edges are coloured red and blue, which already improves results from~\cite{schaudt2019partitioning} (where it is proved that large balanced bipartite graphs can be partitioned into 12 cycles), with a significantly shorter proof. 

 However, Lemma \ref{lmm-Knn-withpath-4cycles} is just the first step towards Theorem~\ref{thm-Knn-4cycles}. The next step is to introduce some tools to prove \Cref{lmm-Knnwithpath-3cycles}. We will also prove a weaker version of \Cref{lmm-Knnwithpath-3cycles}, which will make the proof of \Cref{lmm-Knnwithpath-3cycles} clearer.

\begin{dfn}\label{dfn-specialset}
Let $G$ be a red zigzag graph with $2n$ vertices. For $2\le k\leq n$, the set $\Sleq{k}$ is a \textbf{blue special set} when the following three conditions hold:
\begin{enumerate}[(i)]
\item the subgraph $G[\Sleq{k-2}]$ is either empty or has a blue Hamiltonian cycle,
\item the subgraph $G[\Sleq{k-1}]$ has a blue Hamiltonian cycle, and
\item the subgraph $G[\Sleq{k}]$ has a blue Hamiltonian cycle that uses the edges $x_{k-1}y_{k-1}$ and $x_ky_k$.
\end{enumerate}
\end{dfn}

Furthermore, we make the following remark to shorten the proofs to come.

\begin{rmk}\label{rmk-cycle+cycle}
Let $G$ be a complete bipartite graph and $C_1=(u_1,\ldots,u_s)$ and $C_2=(v_1,\ldots,v_t)$ be two disjoint blue cycles in $G$, where $u_1,v_1\in X$. Then $u_1v_t$ and $v_1u_s$ are edges of $G$. If those edges are both blue, then there is a blue cycle in $G$ that passes through all edges in $C_1$ and $C_2$ except $u_1u_s$ and $v_1v_t$ and covers all vertices in $C_1$ and $C_2$, namely $(u_1,\ldots,u_s,v_1,\ldots,v_t)$ (see Figure~\ref{fig-cycle+cycle}).
\end{rmk}

\begin{figure}[!ht]
\centering
\begin{tikzpicture}[scale=0.8]
        \node (dotsu) at (1.5,4.2) {$\cdots$};

        \node [vertex, label={[label distance=0cm]180:$u_{s-1}$}] at (0.4,4) {};
        \node [vertex, label={[label distance=0cm]0:$u_2$}] at (2.6,4) {};

        \node [vertex, label={[label distance=0cm]180:$u_1$}] (u1) at (0,3) {};
        \node [vertex, label={[label distance=0cm]0:$u_s$}] (us) at (3,3) {};

        \node [vertex, label={[label distance=0cm]180:$v_1$}] (v1) at (0,2) {};
        \node [vertex, label={[label distance=0cm]0:$v_t$}] (vt) at (3,2) {};

        \node [vertex, label={[label distance=0cm]0:$v_2$}] at (2.6,1) {};
        \node [vertex, label={[label distance=0cm]180:$v_{s-1}$}] at (0.4,1) {};

        \node (dotsv) at (1.5,0.8) {$\cdots$};

        \begin{scope}[on background layer]
        \draw [blueedge] (v1)--(us) (u1)--(vt);
        \draw [blueedge] (u1) .. controls (3,4) .. (dotsu);
        \draw [blueedge] (us) .. controls (0,4) .. (dotsu);
        \draw [blueedge] (v1) .. controls (3,1) .. (dotsv);
        \draw [blueedge] (vt) .. controls (0,1) .. (dotsv);
        \end{scope}

        \node at (5, 2.5) {$=$};
        \pgfmathsetmacro{\raio}{1.4}
        \begin{scope}[xshift=7.5cm]
            \begin{scope}[yshift=2.5cm]
                    \draw[blueedge] (20:\raio) arc (20:160:\raio);
                    \draw[blueedge] (200:\raio) arc (200:340:\raio);
                    \node[anchor=center] at (175:\raio) {$\myvdots$};
                    \node[vertex, label={left:$v_1$}] (v1new) at (20:\raio){};
                    \node[vertex, label={below:$v_2$}] at (84:\raio){};
                    \node[vertex, label={left:$v_t$}] (vtnew) at (340:\raio){};
                    \foreach \i in {1,...,5}{
                        \node[vertex] at ({20+320*\i/5}:\raio){};
                    }
            \end{scope}
            \begin{scope}[shift={(3.5,2.5)}]
                    \draw[blueedge] (200:\raio) arc (200:340:\raio);
                    \draw[blueedge] (380:\raio) arc (380:520:\raio);
                    \node[anchor=center] at (5:\raio) {$\myvdots$};
                    \node[vertex, label={right:$u_1$}] (u1new) at (200:\raio){};
                    \node[vertex, label={above:$u_2$}]  at (264:\raio){};
                    \node[vertex, label={right:$u_s$}] (usnew) at (520:\raio){};
                    \foreach \i in {1,...,5}{
                        \node[vertex] at ({200+320*\i/5}:\raio){};
                    }
            \end{scope}
            \draw [blueedge] (v1new)--(usnew) (u1new)--(vtnew);
        \end{scope}
\end{tikzpicture}
\caption{The blue cycle $(u_1,\ldots,u_s,v_1,\ldots,v_t)$ represented in two ways. On the left the layout highlights the two classes of the bipartition.}
\label{fig-cycle+cycle}
\end{figure}

The following lemma shows us that the existence of a blue special set in a red zigzag graph is a sufficient extra condition to make Lemma \ref{lmm-Knnwithpath-3cycles} valid.

\begin{lmm}\label{lmm-Knn-withpath-specialset}
Let $G$ be a red zigzag graph with $2n$ vertices. If $G$ has a \textbf{blue special set}, then $G$ can be vertex-partitioned into at most $3$ monochromatic cycles.
\end{lmm}

\begin{proof}
Let $\Sleq{k}$ be a blue special set in $G$ for some $2\leq k\leq n$. By Definition \ref{dfn-specialset}~(iii), $G[\Sleq{k}]$ has a blue Hamiltonian cycle.
Hence, we may assume that $n > k+2$, since otherwise we would be done.
By Remark \ref{rmk-xiyixiyi+2}(b) and Definition \ref{dfn-specialset}~(iii), we may also assume that the edges $x_{k-1}y_{k+1}$, $x_{k+1}y_{k-1}$, $x_ky_{k+2}$ and $x_{k+2}y_k$ are blue, since otherwise we would be done.

Now, we apply \Cref{lmm-Knn-withpath-4cycles} to the subgraph $G[\Sgt{k}]$. If $G[\Sgt{k}]$ satisfies condition (i), then we may simply take a blue Hamiltonian cycle in $G[\Sleq{k}]$ to obtain a good partition of $G$. 
By Remark~\ref{rmk-cycle+cycle}, if $G[\Sgt{k}]$ satisfies condition (ii), then we  use the edges $x_{k-1}y_{k+1}$ and $x_{k+1}y_{k-1}$ to combine the blue cycle that uses the edge $x_{k+1}y_{k+1}$ from the $3$-cycle partition of $G[\Sgt{k}]$ with the blue Hamiltonian cycle of $G[\Sleq{k}]$ (since it uses the edge $x_{k-1}y_{k-1}$), thereby obtaining a good partition of $G$. Observe that the blue cycle we have just built also passes through the edge $x_ky_k$. Hence, if $G[\Sgt{k}]$ satisfies condition (iii), then first we do the same operation as in the previous case (using edges $x_{k-1}y_{k+1}$ and $x_{k+1}y_{k-1}$) and after that we use the edges $x_ky_{k+2}$ and $x_{k+2}y_k$ to build a blue cycle in $G$ that covers all vertices in $G[\Sleq{k}]$ and all vertices of two different blue cycles of the partition of $G[\Sgt{k}]$ (remember that the edges $x_{k+1}y_{k+1}$ and $x_{k+2}y_{k+2}$ were in different cycles of the partition). Thereby obtaining a desired partition of $G$. In every case, the result follows.
\end{proof}

To prove the next lemma, it will be necessary to make a more involved analysis of the edges compared to what we did in the proofs of Lemmas \ref{lmm-Knn-withpath-4cycles} and \ref{lmm-Knn-withpath-specialset}. Therefore, as a way to facilitate our work and the reader's understanding, we shall make some considerations first.

\begin{dfn}
    An edge $x_iy_j$ in a red zigzag graph is \textbf{even} when $i$ and $j$ have the same parity, i.e., when $i+j$ is even.
\end{dfn}

 Usually, if an even edge with ``low'' values for its indices is red, we can then partition the edges into 3 monochromatic cycles. Otherwise if that edge is blue we can look at other edges. The fact that such an edge is blue helps us little by little to establish a nice blue structure.

\begin{dfn}\label{dfn-evenplait}
Let $G$ be a red zigzag graph with $2n$ vertices. For $1\leq k\leq n$, the subgraph $G[\Sleq{k}]$ is a \textbf{blue even plait} when all its even edges are blue. We denote by $G_k^O$ the subgraph of $G$ formed by vertices of odd indices in $\Sleq{k}$, and $G_k^E$ the one formed by vertices of even indices.
\end{dfn}

\begin{rmk}\label{rmk-evenplait-2cycles}
Let $G$ be a red zigzag graph with $2n$ vertices. If $G[\Sleq{k}]$ is a blue even plait, for $1 \le k\le n$, then $G_k^O$ and $G_k^E$ are complete balanced bipartite graphs whose edges are all blue. In particular, each of them is Hamiltonian, so $G[\Sleq{k}]$ can be partitioned into at most two monochromatic blue cycles.
\end{rmk}

\begin{notation}\label{notation-cycles}
A Hamiltonian path in $G_k^E$ starting in $x_i$ and ending in $y_j$ is denoted $x_i\To[E_k]y_j$ (note that such a path needs to start and finish on different classes of the bipartition).

A path in $G_k^E$ starting in $x_i$ and ending in $x_j$ that passes through all vertices of $G_k^E$ except a given $y_\ell$ is denoted $x_i\To[E_k\setminus y_\ell]x_j$ (note that such a path needs to start and finish on the same class of the bipartition).

We use the same notations ($x_i\To[O_k]y_j$ and $x_i\To[O_k\setminus y_\ell]x_j$) for $G_k^O$.
\end{notation}

Such Hamiltonian or quasi-Hamiltonian paths with specific extremities will be extensively used in the proof of \Cref{lmm-Kn-nsplit-3cycles}. Indeed, we will regularly manage to show that a few edges between $G_k^O$ and $G_k^E$ are blue too. Then by choosing carefully the end-vertices of two $\text{(quasi-)Hamiltonian}$ paths we will be able to combine the two paths into a cycle passing through all the vertices in $G_k^O \cup G_k^E$ except at most two of them, helping us to get a decomposition of the red zigzag graph into at most 3 monochromatic cycles.

\subsection{Main proof}
\label{subsec:mainproof}

After these considerations, we may advance to our next result. As we shall see hereafter, Lemma \ref{lmm-Knnwithpath-3cycles} follows almost directly from the following lemma.

\begin{lmm}\label{lmm-Knn-withpath-evenplait}
Let $G$ be a red zigzag graph with $2n$ vertices which cannot be vertex-partitioned into at most $3$ monochromatic cycles. If the subgraph $G[\Sleq{k}]$ is a blue even plait for some $1 \leq k \leq n-2$, then $G[\Sleq{k+1}]$ is a blue even plait as well (see Definition~\ref{dfn-evenplait}).
\end{lmm}

\begin{proof}
First since $G$ cannot be vertex-partitioned into 3 monochromatic cycles, we can easily see that $n \geq 5$. It is also easy to check that $G[\Sleq{4}]$ is a blue even plait: if the edge $x_1y_1$ (resp. $x_2y_2$) were red, then $G$ (resp. $G\setminus \{x_1y_1\}$) would be Hamiltonian; if any of the edges $x_2y_2$, $x_1y_3$ or $x_3y_1$ were red we could decompose $G$ into a red cycle and one edge; $x_3y_3$ being red would imply a decomposition into a cycle and two paths.\\
Now if $x_2y_4$ or $x_4y_2$ were red we  could decompose $G$ into a red cycle and two edges; if $x_4y_4$ were red then we would decompose into the big red cycle using $x_4y_4$, the blue cycle we just showed the existence ($x_1y_1x_3y_3$) and the edge $x_2y_2$. We have just proved that the  even edges within $G[\Sleq{4}]$ are blue, hence  $G[\Sleq{4}]$ is a blue even plait as we claimed. Therefore we can assume that $k\ge 4$, so that each of $G_k^O$ and $G_k^E$ have at least 4 vertices. Also, Remarks \ref{rmk-xiyixiyi+2}(a) and \ref{rmk-evenplait-2cycles} imply that \textbf{the edge $x_{k+1}y_{k+1}$ is blue}. 

Now, suppose towards a contradiction that $G[\Sleq{k+1}]$ is not a blue even plait. Hence, there is some natural number $j$ less than $k+1$ such that $j$ and $k+1$ have the same parity and at least one of the edges $x_{k+1}y_j$ or $x_jy_{k+1}$ is red. We shall show that this implies that $G$ has a blue special set, which is a contradiction by Lemma~\ref{lmm-Knn-withpath-specialset}.\\

We divide the rest of the proof into two cases depending on the parity of $k$. These cases are treated quite similarly, but at some points they are not absolutely analogous. Therefore, we will guide the reader to understand the analogy whenever possible, and give the full details for both cases when the new one needs bigger adaptations.

\begin{center}
\textbf{Case A:} $k$ is even and $k\ge 4$, therefore $j$ is odd with $1 \leq j < k \leq n-2$.
\end{center}
Note that some claims below will require $k < n-4$, but we will see that these claims are only needed when $k < n-4$: otherwise the proof would end before.
\stepcounter{case}

\begin{clm}\label{clm-x1ykxk-1y2}\mbox{}
\begin{enumerate}[(i)] 
         \item \label{clm-x1ykxk-1y2-itemi} If there exists $1 \leq l \leq k$ odd such that $x_{k+1}y_l$ is red, then $x_1y_k$ and $x_{k-1}y_2$ are not both blue.
         \item \label{clm-x1ykxk-1y2-itemii} If there exists $1 \leq l \leq k$ odd such that $x_ly_{k+1}$ is red, then $x_ky_1$ and $x_2y_{k-1}$ are not both blue.
\end{enumerate}
\end{clm}
\begin{proof}
We will prove only (\ref{clm-x1ykxk-1y2-itemi}) since (\ref{clm-x1ykxk-1y2-itemii}) is analogous. So $x_{k+1}y_l$ is red. Suppose for sake of contradiction that the edges $x_1y_k$ and $x_{k-1}y_2$ are blue (see Figure~\ref{fig-x1ykxk-1y2}). In this case, we shall prove that $\Sleq{k+2}$ is a blue special set of $G$. 

\begin{figure}[!ht]
\centering
\begin{tikzpicture}[yscale=0.7, xscale=1.0]
\node at (-1.5,5) {$G_k^O$};
\node at (4.5,5) {$G_k^O$};
\fill [rounded corners, fill=blue!10] (-1,3.5) rectangle (4,6.5);

\node at (-1.5,8.5) {$G_k^E$};
\node at (4.5,8.5) {$G_k^E$};
\fill [rounded corners, fill=blue!10] (-1,7) rectangle (4,10);

\node [vertex, label={[label distance=0cm]180:$x_1$}] (x1) at (0,4) {};
\node at (0,5) {$\myvdots$};
\node [vertex, label={[label distance=0cm]180:$x_{k-1}$}] (xk-1) at (0,6) {};

\node [vertex, label={[label distance=0cm]180:$x_2$}] (x2) at (0,7.5) {};
\node at (0,8) {$\myvdots$};
\node [vertex, label={[label distance=0cm]180:$x_{l+1}$}] (xl+1) at (0,8.5) {};
\node at (0,9) {$\myvdots$};
\node [vertex, label={[label distance=0cm]180:$x_k$}] (xk) at (0,9.5) {};

\node [vertex, label={[label distance=0cm]180:$x_{k+1}$}] (xk+1) at (0,11) {};
\node [vertex, label={[label distance=0cm]180:$x_{k+2}$}] (xk+2) at (0,12) {};

\node (dots) at (1.5,13) {$\cdots$};

\node [vertex, label={[label distance=0cm]0:$y_1$}] (y1) at (3,4) {};
\node at (3,4.5) {$\myvdots$};
\node [vertex, label={[label distance=0cm]0:$y_l$}] (yl) at (3,5) {};
\node at (3,5.5) {$\myvdots$};
\node [vertex, label={[label distance=0cm]0:$y_{k-1}$}] (yk-1) at (3,6) {};

\node [vertex, label={[label distance=0cm]0:$y_2$}] (y2) at (3,7.5) {};
\node at (3,8.5) {$\myvdots$};
\node [vertex, label={[label distance=0cm]0:$y_k$}] (yk) at (3,9.5) {};

\node [vertex, label={[label distance=0cm]0:$y_{k+1}$}] (yk+1) at (3,11) {};
\node [vertex, label={[label distance=0cm]0:$y_{k+2}$}] (yk+2) at (3,12) {};

\draw [rededge] (yl)--(xk+1)--(yk+2);
\draw [rededge] (xk+1)--(yk) (xk)--(yk+1)--(xk+2); 
\draw [rededge] (xk+2) .. controls (3,13) .. (dots);
\draw [rededge] (yk+2) .. controls (0,13) .. (dots);
\draw [blueedge] (x1)--(yk) (xk-1)--(y2) (xk+1)--(yk+1);
\draw[gray, dashed] (xk)--(yl) (xk+1)--(yk-1) (xl+1)--(yk+1);
\end{tikzpicture}
\caption{The case where the edges $x_1y_k$ and $x_{k-1}y_2$ are blue.}
\label{fig-x1ykxk-1y2}
\end{figure}
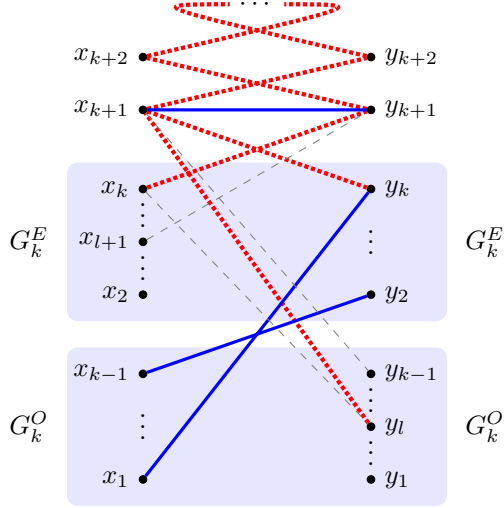

First, let us prove that the edge $x_ky_l$ is blue. Indeed, otherwise we could decompose $G$ into the red cycle $(x_k,y_{k+1}Px_{k+1},y_l)$ that covers $\Sgeq{k+1}\cup \{x_k, y_l\}$ and a blue cycle that covers $\Sleq{k} \setminus \{x_k, y_l\}$. More precisely:
\[
    y_2\To[E_k\setminus x_k] y_k \To x_1 \To[O_k\setminus y_l] x_{k-1} \To y_2.
\]

Now, we can construct a blue Hamiltonian cycle for $G[\Sleq{k}]$: $y_k\To[E_k] x_k \To y_l \To[O_k] x_1 \To y_k$. This cycle was made by taking one cycle in $E_k$, one in $O_k$, removing an edge to each, and linking the two obtained paths with two specific edges. We will use this technique a lot in this proof.
This is our first step towards proving that $\Sleq{k+2}$ is a blue special set of $G$.

The edge $x_{k+1}y_{k-1}$ is blue, otherwise $G$ could be partitioned into the red cycle $(x_k,y_{k+1}Px_{k+1},y_{k-1})$ (recall that $x_ky_{k-1}$ is also an edge of $P$) and (similarly to above) a blue cycle that covers all vertices in $\Sleq{k}$ except $x_k$ and $y_{k-1}$. 
The edge $x_{l+1}y_{k+1}$ is blue too, since otherwise $G$ could similarly be vertex-partitioned into the red cycle $(x_{l+1},y_{k+1}Px_{k+1},y_l)$ and a blue cycle that 
covers all vertices in $\Sleq{k}$ except $x_{l+1}$ and $y_l$. Therefore, we see that the subgraph $G[\Sleq{k+1}]$ has a blue Hamiltonian cycle: \[
        C = (x_{k+1}\To y_{k+1} \To x_{l+1} \To[E_k] y_2 \To x_{k-1} \To[O_k] y_{k-1} \To x_{k+1}).
\]
By Remark \ref{rmk-xiyixiyi+2}(a), it follows that the edge $x_{k+2}y_{k+2}$ is blue.

Next, the edge $x_{k+2}y_k$ is blue, otherwise $G$ could be vertex-partitioned into the edge-cycle $(x_{l+1},y_{k+1})$, the red cycle $(x_{k+1},y_{k+2}Px_{k+2},y_k)$,  and a blue cycle that covers $\Sleq{k}\setminus \{x_{l+1}, y_k\}$. Note that such blue cycle does not use $y_k$ so it does not benefit from the blue edge $x_1y_k$ but rather from $x_ky_l$. Use the edges in $G_k^E$ to go from $y_2$ to $x_{k}$ passing through all its vertices except $x_{l+1}$ and $y_k$, add edge $x_ky_l$ go through all vertices in $G_k^O$ finishing in $x_{k-1}$ and add the edge $x_{k-1}y_2$.
 The edge $x_ky_{k+2}$ is blue too, since otherwise $G$ could be vertex-partitioned into the red cycle $(x_k,y_{k+2}Px_{k+2},y_{k+1})$, the edge-cycle $(x_{k+1},y_{k-1})$ and a blue cycle that covers all vertices of $\Sleq{k}\setminus \{x_k, y_{k-1}\}$ (using the edges $x_1y_k$ and $x_{k-1}y_2$). 

 Hence, we see that the subgraph $G[\Sleq{k+2}]$ has a blue Hamiltonian cycle that uses the edges $x_{k+1}y_{k+1}$ and $x_{k+2}y_{k+2}$, which can be built by the following manner: take the blue Hamiltonian cycle $C$ of $G[\Sleq{k+1}]$ that we have just built. We can require it to contain the edge $x_ky_k$. Then consider the cycle $C' = (C - x_ky_k) + x_ky_{k+2}x_{k+2}y_k$.
Therefore $\Sleq{k+2}$ is a blue special set of $G$, as needed.
\end{proof}

\begin{clm}\label{clm-xk+1y1}
The edges $x_{k+1}y_1$ and $x_1y_{k+1}$ are blue.
\end{clm}
\begin{proof}
Suppose towards a contradiction that the edge $x_{k+1}y_1$ is red (see Figure~\ref{fig-xk+1y1}). In this case, the vertices of $G$ can be partitioned into the red cycle $(x_{k+1}Py_1)$ and the red path $(x_1,y_2Px_{k-1},y_k)$. By Claim \ref{clm-x1ykxk-1y2}(\ref{clm-x1ykxk-1y2-itemi}), with $l=1$, we know that the edges $x_1y_k$ and $x_{k-1}y_2$ are not both blue. Hence, $G$ can be vertex-partitioned into at most $3$ monochromatic cycles, a contradiction. Thus, the edge $x_{k+1}y_1$ is blue. Analogously, the edge $x_1y_{k+1}$ is blue too, by Claim \ref{clm-x1ykxk-1y2}(\ref{clm-x1ykxk-1y2-itemii}). Thus, the result follows.

\begin{figure}[!ht]
\centering
\begin{tikzpicture}[scale=1.0]
\node [vertex, label={[label distance=0cm]270:$x_1$}] (x1) at (0,0) {};
\node [vertex, label={[label distance=0cm]270:$y_2$}] (y2) at (1,0) {};
\node (dotsx) at (2,0) {$\cdots$};
\node [vertex, label={[label distance=0cm]270:$x_{k-1}$}] (xk-1) at (3,0) {};
\node [vertex, label={[label distance=0cm]270:$y_k$}] (yk) at (4,0) {};
\node [vertex, label={[label distance=0cm]270:$x_{k+1}$}] (xk+1) at (5,0) {};
\node (dotsn) at (8,0) {$\cdots$};
\node [vertex, label={[label distance=0cm]270:$y_1$}] (y1) at (11,0) {};

\draw [rededge] (x1)--(y2)--(dotsx)--(xk-1)--(yk)--(xk+1)--(dotsn)--(y1);
\draw [rededge] (xk+1) to [out=30, in=150] (y1);

\end{tikzpicture}
\caption{The case where the edge $x_{k+1}y_1$ is red.}
\label{fig-xk+1y1}
\end{figure}
\end{proof}

\begin{clm}\label{clm-xk+2yk+2-A}
The edge $x_{k+2}y_{k+2}$ is blue.
\end{clm}

\begin{proof}
By \Cref{clm-xk+1y1}, we see that the subgraph $G[\Sleq{k+1}]$ can be vertex-partitioned into a blue Hamiltonian cycle of $G_k^E$ and a blue cycle that covers $x_{k+1}$, $y_{k+1}$ and all vertices in $G_k^O$, which can be built by the following manner: take a blue Hamiltonian cycle of $G_k^O$ that uses the edge $x_1y_1$; and use the path $(x_1,y_{k+1},x_{k+1},y_1)$ to extend it. Hence, by Remark \ref{rmk-xiyixiyi+2}(a), the result follows.
\end{proof}

Remember that there is some odd $j$ such that $x_{k+1}y_j$ or $x_jy_{k+1}$ is red. In what remains, without loss of generality, \textbf{we assume that the edge $x_{k+1}y_j$ is red.}

Note: with that we loose the symmetry between the sets $X$ and $Y$. We did not make this assumption earlier, because we need Claim~\ref{clm-xk+1y1} to hold regardless of which edge between $x_{k+1}y_j$ and $x_jy_{k+1}$ is red.

By \Cref{clm-xk+1y1}, we know that $j \neq 1$, hence all the following claims will be on the assumption that $2 \leq j$. Besides, if $k = n-2$ then the edge $x_{k+2}y_{k+2} = x_ny_n$ belongs to the red path $P$, hence is not blue, contradicting \Cref{clm-xk+2yk+2-A}. Thus if $k = n-2$ the proof ends here. Therefore, from here we have:
\[
        2\leq j < k \le n-3.
\]
This allows us to use indices $j-1$ and $k+3$ (we have not used them so far).

\begin{clm}\label{clm-xj+1y1x2yj+2}
The edges $x_{j+1}y_1$ and $x_2y_{j+2}$ are blue.\\ In particular, when $j=k-1$, this means that $x_ky_1$ and $x_2y_{k+1}$ are blue.
\end{clm}
\begin{proof}
Suppose towards a contradiction that the edge $x_{j+1}y_1$ is red (see Figure~\ref{fig-xj+1y1}). In this case, $G$ can be vertex-partitioned into the red cycle $(x_{k+1}Px_{j+1}, y_1Py_j)$ and the red path $(x_1,y_2Px_{k-1},y_k)$. By Claim~\ref{clm-x1ykxk-1y2}(\ref{clm-x1ykxk-1y2-itemi}), we know that the edges $x_1y_k$ and $x_{k-1}y_2$ are not both blue. Hence, we see that $G$ can be vertex-partitioned into at most $3$ monochromatic cycles, a contradiction. Thus, the edge $x_{j+1}y_1$ is blue.

\begin{figure}[!ht]
\centering
\begin{tikzpicture}[scale=1.0]
\node [vertex, label={[label distance=0cm]270:$x_1$}] (x1) at (0,0) {};
\node [vertex, label={[label distance=0cm]270:$y_2$}] (y2) at (1,0) {};
\node (dotsx) at (2,0) {$\cdots$};
\node [vertex, label={[label distance=0cm]270:$x_{k-1}$}] (xk-1) at (3,0) {};
\node [vertex, label={[label distance=0cm]270:$y_k$}] (yk) at (4,0) {};
\node [vertex, label={[label distance=0cm]270:$x_{k+1}$}] (xk+1) at (5,0) {};
\node (dotsn) at (6.5,0) {$\cdots$};
\node [vertex, label={[label distance=0cm]270:$x_{j+1}$}] (xj+1) at (8,0) {};
\node [vertex, label={[label distance=0cm]270:$y_j$}] (yj) at (9,0) {};
\node (dotsy) at (10,0) {$\cdots$};
\node [vertex, label={[label distance=0cm]270:$y_1$}] (y1) at (11,0) {};

\draw [rededge] (x1)--(y2)--(dotsx)--(xk-1)--(yk)--(xk+1)--(dotsn)--(xj+1)--(yj)--(dotsy)--(y1);
\draw [rededge] (xj+1) to [out=30, in=150] (y1);
\draw [rededge] (xk+1) to [out=30, in=150] (yj);
\end{tikzpicture}
\caption{The case where the edge $x_{j+1}y_1$ is red.}
\label{fig-xj+1y1}
\end{figure}

Now, suppose for sake of contradiction that the edge $x_2y_{j+2}$ is red (see Figure~\ref{fig-x2yj+2}). In this case, the edge $x_1y_k$ is blue, since otherwise $G$ could be vertex-partitioned into the red cycles $(x_{k+1}Py_{j+2},x_2Py_j)$ and $(x_1,y_2Px_{k-1},y_k)$ and the edge-cycle $(x_{j+1},y_1)$. By Claim \ref{clm-x1ykxk-1y2}(\ref{clm-x1ykxk-1y2-itemi}), it follows that the edge $x_{k-1}y_2$ is red. Observe that the edges $x_1y_1$ and $x_{j+1}y_k$ are even and, therefore, blue. Hence, we see that $G$ can be vertex-partitioned into the red cycles $(x_{k+1}Py_{j+2},x_2Py_j)$ and $(y_2Px_{k-1})$ and the blue cycle $(x_1,y_k,x_{j+1},y_1)$, a contradiction. Thus, the edge $x_2y_{j+2}$ is blue.

\begin{figure}[!ht]
\centering
        \begin{tikzpicture}[scale=0.88]
        \node [vertex, label={[label distance=0cm]270:$x_1$}] (x1) at (0,0) {};
        \node [vertex, label={[label distance=0cm]270:$y_2$}] (y2) at (1,0) {};
        \node (dotsx) at (2,0) {$\cdots$};
        \node [vertex, label={[label distance=0cm]270:$x_{k-1}$}] (xk-1) at (3,0) {};
        \node [vertex, label={[label distance=0cm]270:$y_k$}] (yk) at (4,0) {};
        \node [vertex, label={[label distance=0cm]270:$x_{k+1}$}] (xk+1) at (5,0) {};
        \node (dotsn) at (6.5,0) {$\cdots$};
        \node [vertex, label={[label distance=0cm]270:$y_{j+2}$}] (yj+2) at (8,0) {};
        \node [vertex, label={[label distance=0cm]270:$x_{j+1}$}] (xj+1) at (9,0) {};
        \node [vertex, label={[label distance=0cm]270:$y_j$}] (yj) at (10,0) {};
        \node (dotsy) at (11,0) {$\cdots$};
        \node [vertex, label={[label distance=0cm]270:$x_2$}] (x2) at (12,0) {};
        \node [vertex, label={[label distance=0cm]270:$y_1$}] (y1) at (13,0) {};

        \draw [rededge] (x1)--(y2)--(dotsx)--(xk-1)--(yk)--(xk+1)--(dotsn)--(yj+2)--(xj+1)--(yj)--(dotsy)--(x2)--(y1);
        \draw [rededge] (yj+2) to [out=30, in=150] (x2);
        \draw [rededge] (xk+1) to [out=30, in=150] (yj);

        \draw [blueedge] (x1) to [out=20, in=160] (y1);
        \draw [blueedge] (yk) to [out=20, in=160] (xj+1);
        \draw [blueedge] (xj+1) to [out=20, in=160] (y1);

        \end{tikzpicture}
        \caption{The case where the edge $x_2y_{j+2}$ is red.}
        \label{fig-x2yj+2}
\end{figure}
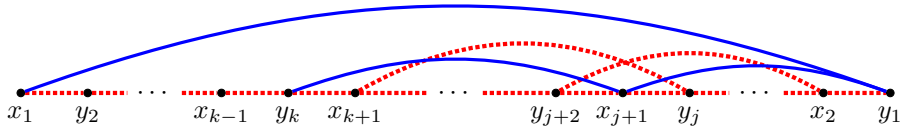
\end{proof}

\begin{clm}\label{clm-xjyk+2xk+2yk-A}
The edges $x_jy_{k+2}$ and $x_{k+2}y_k$ are blue.
\end{clm}
\begin{proof}
Suppose towards a contradiction that the edge $x_jy_{k+2}$ is red (see Figure~\ref{fig-xjyk+2}). In this case, $G$ can be vertex-partitioned into the red cycle $(x_jPx_{k+1},y_jPy_{k+2})$ and the subgraph $G[\Sleq{j-1}]$. But, observe that the subgraph $G[\Sleq{j-1}]$ is a blue even plait. Hence, by \Cref{rmk-evenplait-2cycles}, we see that $G$ can be vertex-partitioned into $3$ monochromatic cycles, a contradiction. Thus, the edge $x_jy_{k+2}$ is blue.

\begin{figure}[!ht]
\centering
\begin{subfigure}[!ht]{1.0\textwidth}
\centering
\begin{tikzpicture}[scale=1.0]
\node [vertex, label={[label distance=0cm]270:$x_j$}] (xj) at (4,0) {};
\node (dotsx) at (5,0) {$\cdots$};
\node [vertex, label={[label distance=0cm]270:$x_{k+1}$}] (xk+1) at (6,0) {};
\node [vertex, label={[label distance=0cm]270:$y_{k+2}$}] (yk+2) at (7,0) {};
\node (dotsn) at (9,0) {$\cdots$};
\node [vertex, label={[label distance=0cm]270:$y_j$}] (yj) at (11,0) {};

\draw [rededge] (xj)--(dotsx)--(xk+1)--(yk+2)--(dotsn)--(yj);
\draw [rededge] (xj) to [out=30, in=150] (yk+2);
\draw [rededge] (xk+1) to [out=30, in=150] (yj);
\end{tikzpicture}
\end{subfigure}

\vspace{0.3cm}

\begin{subfigure}[!ht]{1.0\textwidth}
\centering
\begin{tikzpicture}[yscale=0.7, xscale=1.0]
\node at (-1.5,4.5) {$G_{j-1}^O$};
\node at (4.5,4.5) {$G_{j-1}^O$};
\fill [rounded corners, fill=blue!10] (-1,3.5) rectangle (4,5.5);

\node at (-1.5,7) {$G_{j-1}^E$};
\node at (4.5,7) {$G_{j-1}^E$};
\fill [rounded corners, fill=blue!10] (-1,6) rectangle (4,8);

\node [vertex, label={[label distance=0cm]180:$x_1$}] (x1) at (0,4) {};
\node at (0,4.5) {$\myvdots$};
\node [vertex, label={[label distance=0cm]180:$x_{j-2}$}] (xj-2) at (0,5) {};

\node [vertex, label={[label distance=0cm]180:$x_2$}] (x2) at (0,6.5) {};
\node at (0,7) {$\myvdots$};
\node [vertex, label={[label distance=0cm]180:$x_{j-1}$}] (xj-1) at (0,7.5) {};

\node [vertex, label={[label distance=0cm]0:$y_1$}] (y1) at (3,4) {};
\node at (3,4.5) {$\myvdots$};
\node [vertex, label={[label distance=0cm]0:$y_{j-2}$}] (yj-2) at (3,5) {};

\node [vertex, label={[label distance=0cm]0:$y_2$}] (y2) at (3,6.5) {};
\node at (3,7) {$\myvdots$};
\node [vertex, label={[label distance=0cm]0:$y_{j-1}$}] (yj-1) at (3,7.5) {};

\end{tikzpicture}
\end{subfigure}
\caption{The case where the edge $x_jy_{k+2}$ is red.}
\label{fig-xjyk+2}
\end{figure}


Now, suppose for sake of contradiction that $x_{k+2}y_k$ is red. Take the red cycle $(x_{k+1},y_{k+2}Px_{k+2},y_k)$, that covers $\Sgeq{k+2} \cup \{x_{k+1}, y_k\}$. We claim that the remaining vertices of $G$ can be vertex-partitioned into at most $3$ blue cycles. Indeed, if $j=k-1$, by \Cref{clm-xj+1y1x2yj+2} the edges $x_ky_1$ and $x_2y_{k+1}$ are blue and by \Cref{clm-xk+1y1} $x_1 y_{k+1}$ is blue  (see Figure~\ref{fig-xk+2yk-a}).
Take only the blue cycle
  $$x_2\To[E_k\setminus y_k] x_k\To y_1\To[O_k] x_1 \To y_{k+1} \To x_2.$$

On the other hand, if $j<k-1$, then $y_{j+2}\in G_k^O$ (see Figure~\ref{fig-xk+2yk-b}).
Again by \Cref{clm-xj+1y1x2yj+2} and \Cref{clm-xk+1y1}, we can form the blue cycle 
        $$x_2\To[E_k\setminus y_k] x_{j+1} \To y_1 \To[O_k\setminus x_1] y_{j+2} \To x_2.$$
We also take the edge-cycle $(x_1,y_{k+1})$.
\begin{figure}[!ht]
\centering
\begin{subfigure}[!ht]{0.5\textwidth}
\centering
\begin{tikzpicture}[yscale=0.7, xscale=1.0]
\node at (-1.5,5) {$G_k^O$};
\fill [rounded corners, fill=blue!10] (-1,3.5) rectangle (4,6.5);

\node at (-1.5,8.5) {$G_k^E$};
\fill [rounded corners, fill=blue!10] (-1,7) rectangle (4,10);

\node [vertex, label={[label distance=0cm]180:$x_1$}] (x1) at (0,4) {};
\node at (0,5) {$\myvdots$};
\node [vertex, label={[label distance=0cm]180:$x_{k-1}$}] (xk-1) at (0,6) {};

\node [vertex, label={[label distance=0cm]180:$x_2$}] (x2) at (0,7.5) {};
\node at (0,8.5) {$\myvdots$};
\node [vertex, label={[label distance=0cm]180:$x_k$}] (xk) at (0,9.5) {};

\node [vertex, label={[label distance=0cm]180:$x_{k+1}$}] (xk+1) at (0,11) {};
\node [vertex, label={[label distance=0cm]180:$x_{k+2}$}] (xk+2) at (0,12) {};

\node (dots) at (1.5,13) {$\cdots$};

\node [vertex, label={[label distance=0cm]0:$y_1$}] (y1) at (3,4) {};
\node at (3,5) {$\myvdots$};
\node [vertex, label={[label distance=0cm]0:$y_{k-1}$}] (yk-1) at (3,6) {};

\node [vertex, label={[label distance=0cm]0:$y_2$}] (y2) at (3,7.5) {};
\node at (3,8.5) {$\myvdots$};
\node [vertex, label={[label distance=0cm]0:$y_k$}] (yk) at (3,9.5) {};

\node [vertex, label={[label distance=0cm]0:$y_{k+1}$}] (yk+1) at (3,11) {};
\node [vertex, label={[label distance=0cm]0:$y_{k+2}$}] (yk+2) at (3,12) {};

\draw [rededge] (xk+2)--(yk)--(xk+1)--(yk+2);
\draw [rededge] (xk+2) .. controls (3,13) .. (dots);
\draw [rededge] (yk+2) .. controls (0,13) .. (dots);
\draw [blueedge] (x1)--(yk+1)--(x2) (xk)--(y1);

\end{tikzpicture}
\caption{$j=k-1$}
\label{fig-xk+2yk-a}
\end{subfigure}%
\begin{subfigure}[!ht]{0.5\textwidth}
\centering
\begin{tikzpicture}[yscale=0.7, xscale=1.0]
\node at (4.5,5) {$G_k^O$};
\fill [rounded corners, fill=blue!10] (-1,3.5) rectangle (4,6.5);

\node at (4.5,8.5) {$G_k^E$};
\fill [rounded corners, fill=blue!10] (-1,7) rectangle (4,10);

\node [vertex, label={[label distance=0cm]180:$x_1$}] (x1) at (0,4) {};
\node at (0,5) {$\myvdots$};
\node [vertex, label={[label distance=0cm]180:$x_{k-1}$}] (xk-1) at (0,6) {};

\node [vertex, label={[label distance=0cm]180:$x_2$}] (x2) at (0,7.5) {};
\node at (0,8) {$\myvdots$};
\node [vertex, label={[label distance=0cm]180:$x_{j+1}$}] (xj+1) at (0,8.5) {};
\node at (0,9) {$\myvdots$};
\node [vertex, label={[label distance=0cm]180:$x_k$}] (xk) at (0,9.5) {};

\node [vertex, label={[label distance=0cm]180:$x_{k+1}$}] (xk+1) at (0,11) {};
\node [vertex, label={[label distance=0cm]180:$x_{k+2}$}] (xk+2) at (0,12) {};

\node (dots) at (1.5,13) {$\cdots$};

\node [vertex, label={[label distance=0cm]0:$y_1$}] (y1) at (3,4) {};
\node at (3,4.5) {$\myvdots$};
\node [vertex, label={[label distance=0cm]0:$y_{j+2}$}] (yj+2) at (3,5) {};
\node at (3,5.5) {$\myvdots$};
\node [vertex, label={[label distance=0cm]0:$y_{k-1}$}] (yk-1) at (3,6) {};

\node [vertex, label={[label distance=0cm]0:$y_2$}] (y2) at (3,7.5) {};
\node at (3,8.5) {$\myvdots$};
\node [vertex, label={[label distance=0cm]0:$y_k$}] (yk) at (3,9.5) {};

\node [vertex, label={[label distance=0cm]0:$y_{k+1}$}] (yk+1) at (3,11) {};
\node [vertex, label={[label distance=0cm]0:$y_{k+2}$}] (yk+2) at (3,12) {};

\draw [rededge] (xk+2)--(yk)--(xk+1)--(yk+2);
\draw [rededge] (xk+2) .. controls (3,13) .. (dots);
\draw [rededge] (yk+2) .. controls (0,13) .. (dots);
\draw [blueedge] (x1)--(yk+1) (y1)--(xj+1) (x2)--(yj+2);

\end{tikzpicture}
\caption{$j<k-1$}
\label{fig-xk+2yk-b}
\end{subfigure}

\caption{The case where the edge $x_{k+2}y_k$ is red.}
\label{fig-xk+2yk}
\end{figure}
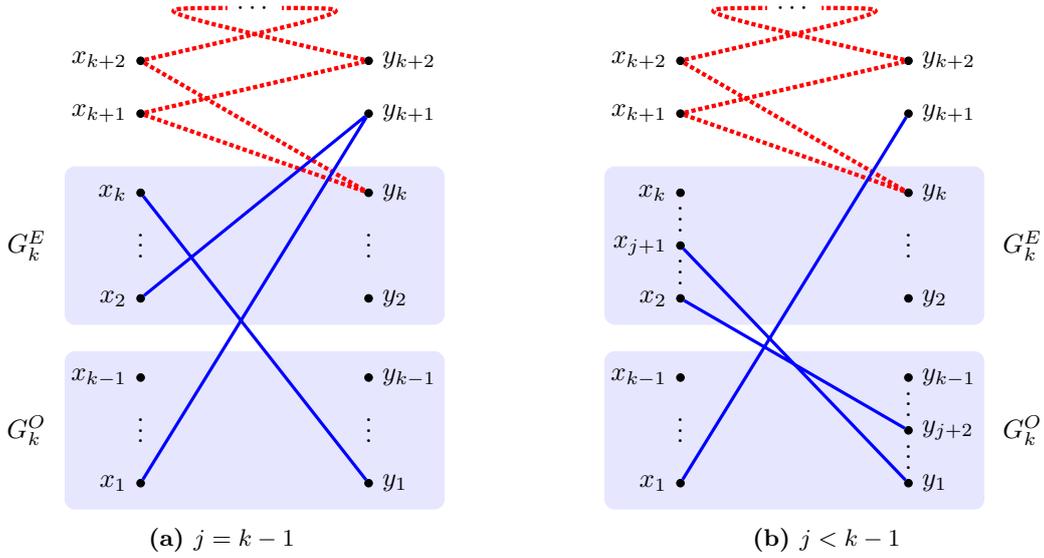
\end{proof}

\begin{clm}\label{clm-GSk+2-A}
The subgraph $G[\Sleq{k+2}]$ has a blue Hamiltonian cycle that uses the edges $x_{k+1}y_{k+1}$, $x_{k+2}y_{k+2}$ and $x_{k+2}y_k$. Furthermore, the edge $x_{k+3}y_{k+3}$ is blue.
\end{clm}
\begin{proof}
Let us build such a cycle (see \Cref{fig-GSk+2}). We will use the blue edges $x_{j+1}y_1$ (\Cref{clm-xj+1y1x2yj+2}), $x_{k+1}y_1$ and $x_1y_{k+1}$ (\Cref{clm-xk+1y1}), $x_{k+1}y_{k+1}$ (Remark~\ref{rmk-xiyixiyi+2}(a)), $x_{k+2}y_k$ (\Cref{clm-xjyk+2xk+2yk-A}), $x_{k+2}y_{k+2}$ (\Cref{clm-xk+2yk+2-A}), and $x_jy_{k+2}$ (\Cref{clm-xjyk+2xk+2yk-A}).

First, take the cycle
$$C = x_{j+1}\To[E_k] y_k \To x_{k+2} \To y_{k+2} \To x_j \To[O_k\setminus y_1] x_1 \To y_1 \To x_{j+1}.$$
It is a blue cycle covering all vertices of $\Sleq{k}\cup \{x_{k+2},y_{k+2}\}$ and it uses the edge $x_1y_1$. So $C' = C - x_1y_1 +x_1y_{k+1}x_{k+1}y_1$ is a blue cycle which satisfies the first part of the claim. By Remark \ref{rmk-xiyixiyi+2}(a), it follows that the edge $x_{k+3}y_{k+3}$ is blue.
\end{proof}

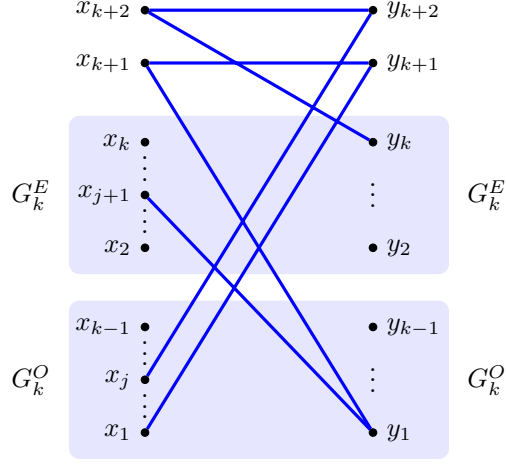
\begin{figure}[!ht]
\centering
\begin{tikzpicture}[yscale=0.7, xscale=1.0]
\node at (-1.5,5) {$G_k^O$};
\node at (4.5,5) {$G_k^O$};
\fill [rounded corners, fill=blue!10] (-1,3.5) rectangle (4,6.5);

\node at (-1.5,8.5) {$G_k^E$};
\node at (4.5,8.5) {$G_k^E$};
\fill [rounded corners, fill=blue!10] (-1,7) rectangle (4,10);

\node [vertex, label={[label distance=0cm]180:$x_1$}] (x1) at (0,4) {};
\node at (0,4.5) {$\myvdots$};
\node [vertex, label={[label distance=0cm]180:$x_j$}] (xj) at (0,5) {};
\node at (0,5.5) {$\myvdots$};
\node [vertex, label={[label distance=0cm]180:$x_{k-1}$}] (xk-1) at (0,6) {};

\node [vertex, label={[label distance=0cm]180:$x_2$}] (x2) at (0,7.5) {};
\node at (0,8) {$\myvdots$};
\node [vertex, label={[label distance=0cm]180:$x_{j+1}$}] (xj+1) at (0,8.5) {};
\node at (0,9) {$\myvdots$};
\node [vertex, label={[label distance=0cm]180:$x_k$}] (xk) at (0,9.5) {};

\node [vertex, label={[label distance=0cm]180:$x_{k+1}$}] (xk+1) at (0,11) {};
\node [vertex, label={[label distance=0cm]180:$x_{k+2}$}] (xk+2) at (0,12) {};

\node [vertex, label={[label distance=0cm]0:$y_1$}] (y1) at (3,4) {};
\node at (3,5) {$\myvdots$};
\node [vertex, label={[label distance=0cm]0:$y_{k-1}$}] (yk-1) at (3,6) {};

\node [vertex, label={[label distance=0cm]0:$y_2$}] (y2) at (3,7.5) {};
\node at (3,8.5) {$\myvdots$};
\node [vertex, label={[label distance=0cm]0:$y_k$}] (yk) at (3,9.5) {};

\node [vertex, label={[label distance=0cm]0:$y_{k+1}$}] (yk+1) at (3,11) {};
\node [vertex, label={[label distance=0cm]0:$y_{k+2}$}] (yk+2) at (3,12) {};

\draw [blueedge] (x1)--(yk+1)--(xk+1)--(y1)--(xj+1) (xj)--(yk+2)--(xk+2)--(yk);
\end{tikzpicture}
\caption{A blue Hamiltonian cycle of $G[\Sleq{k+2}]$.}
\label{fig-GSk+2}
\end{figure}

As before, if $k = n-3$ the proof ends here since we have just proved that the edge $x_ny_n$ is blue, whereas it belongs to the red path $P$. This is a contradiction. Hence, in the remaining claims the following inequality holds \[
k \le n-4.      
\]
This allows us to use indices up to $k+4$.

\begin{clm}\label{clm-xjyk+1-A}
The edge $x_jy_{k+1}$ is blue.
\end{clm}

\begin{proof}
Suppose for sake of contradiction that the edge $x_jy_{k+1}$ is red. In this case, we shall prove that $\Sleq{k+2}$ is a blue special set of $G$. Now \Cref{clm-xk+1y1} tells us that the edges $x_{k+1}y_1$ and $x_1y_{k+1}$ are blue. \Cref{clm-xj+1y1x2yj+2} tells us that the edges $x_{j+1}y_1$ and $x_2y_{j+2}$ are blue too. Also, since here we assume that the edge $x_jy_{k+1}$ is red we can use the `alternate' version of \Cref{clm-xj+1y1x2yj+2} consisting in exchanging $X$ and $Y$, therefore the edges $x_1y_{j+1}$ and $x_{j+2}y_2$ are also blue. Figure~\ref{fig-xjyk+1} shows all those edges.

On the one hand, if $j=k-1$ then this means that the edges $x_2y_{k+1}$ and $x_{k+1}y_2$ are blue.  Then $x_2 \To[E_k] y_2 \To x_{k+1} \To y_1 \To[O_k] x_1 \To y_{k+1} \To x_2$ is a blue Hamiltonian cycle of $G[\Sleq{k+1}]$.
On the other hand if $j<k-1$ then $x_{j+2},y_{j+2}\in G_k^O$. Start with the following blue Hamiltonian cycle of $G[\Sleq{k}]$:
$C = x_2 \To[E_k] y_2 \To x_{j+2} \To[O_k] y_{j+2} \To x_2$. We can require it to use the edge $x_1y_1$ ($G_k^O$ is a blue bipartite complete graph, hence $x_{j+2} \To[O_k] y_{j+2}$ can be required to use $x_1y_1$). Therefore $C' = C-x_1y_1 +x_1y_{x+1}x_{k+1}y_1$ is a blue Hamiltonian cycle of $G[\Sleq{k+1}]$.

In both cases, by Claim \ref{clm-GSk+2-A}, we see that $\Sleq{k+2}$ is a blue special set of $G$, a contradiction. Thus, the result follows.
\end{proof}

\begin{figure}[!ht]
\centering
\begin{subfigure}[!ht]{0.5\textwidth}
\centering
\begin{tikzpicture}[yscale=0.7, xscale=1.0]
        \node at (-1.5,5) {$G_k^O$};
        \fill [rounded corners, fill=blue!10] (-1,3.5) rectangle (4,6.5);

        \node at (-1.5,8.5) {$G_k^E$};
        \fill [rounded corners, fill=blue!10] (-1,7) rectangle (4,10);

        \node [vertex, label={[label distance=0cm]180:$x_1$}] (x1) at (0,4) {};
        \node at (0,5) {$\myvdots$};
        \node [vertex, label={[label distance=0cm]180:$x_{k-1}$}] (xk-1) at (0,6) {};

        \node [vertex, label={[label distance=0cm]180:$x_2$}] (x2) at (0,7.5) {};
        \node at (0,8.5) {$\myvdots$};
        \node [vertex, label={[label distance=0cm]180:$x_k$}] (xk) at (0,9.5) {};

        \node [vertex, label={[label distance=0cm]180:$x_{k+1}$}] (xk+1) at (0,11) {};

        \node [vertex, label={[label distance=0cm]0:$y_1$}] (y1) at (3,4) {};
        \node at (3,5) {$\myvdots$};
        \node [vertex, label={[label distance=0cm]0:$y_{k-1}$}] (yk-1) at (3,6) {};

        \node [vertex, label={[label distance=0cm]0:$y_2$}] (y2) at (3,7.5) {};
        \node at (3,8.5) {$\myvdots$};
        \node [vertex, label={[label distance=0cm]0:$y_k$}] (yk) at (3,9.5) {};

        \node [vertex, label={[label distance=0cm]0:$y_{k+1}$}] (yk+1) at (3,11) {};

        \draw [blueedge] (yk)--(x1)--(yk+1) (xk+1)--(y1)--(xk) (xk+1)--(y2) (yk+1)--(x2);
        \end{tikzpicture}
\caption{$j=k-1$}
\label{fig-xjyk+1-a}
\end{subfigure}%
\begin{subfigure}[!ht]{0.5\textwidth}
\centering
\begin{tikzpicture}[yscale=0.7, xscale=1.0]
        \node at (4.5,5) {$G_k^O$};
        \fill [rounded corners, fill=blue!10] (-1,3.5) rectangle (4,6.5);

        \node at (4.5,8.5) {$G_k^E$};
        \fill [rounded corners, fill=blue!10] (-1,7) rectangle (4,10);

        \node [vertex, label={[label distance=0cm]180:$x_1$}] (x1) at (0,4) {};
        \node at (0,4.5) {$\myvdots$};
        \node [vertex, label={[label distance=0cm]180:$x_{j+2}$}] (xj+2) at (0,5) {};
        \node at (0,5.5) {$\myvdots$};
        \node [vertex, label={[label distance=0cm]180:$x_{k-1}$}] (xk-1) at (0,6) {};

        \node [vertex, label={[label distance=0cm]180:$x_2$}] (x2) at (0,7.5) {};
        \node at (0,8) {$\myvdots$};
        \node [vertex, label={[label distance=0cm]180:$x_{j+1}$}] (xj+1) at (0,8.5) {};
        \node at (0,9) {$\myvdots$};
        \node [vertex, label={[label distance=0cm]180:$x_k$}] (xk) at (0,9.5) {};

        \node [vertex, label={[label distance=0cm]180:$x_{k+1}$}] (xk+1) at (0,11) {};

        \node [vertex, label={[label distance=0cm]0:$y_1$}] (y1) at (3,4) {};
        \node at (3,4.5) {$\myvdots$};
        \node [vertex, label={[label distance=0cm]0:$y_{j+2}$}] (yj+2) at (3,5) {};
        \node at (3,5.5) {$\myvdots$};
        \node [vertex, label={[label distance=0cm]0:$y_{k-1}$}] (yk-1) at (3,6) {};

        \node [vertex, label={[label distance=0cm]0:$y_2$}] (y2) at (3,7.5) {};
        \node at (3,8) {$\myvdots$};
        \node [vertex, label={[label distance=0cm]0:$y_{j+1}$}] (yj+1) at (3,8.5) {};
        \node at (3,9) {$\myvdots$};
        \node [vertex, label={[label distance=0cm]0:$y_k$}] (yk) at (3,9.5) {};

        \node [vertex, label={[label distance=0cm]0:$y_{k+1}$}] (yk+1) at (3,11) {};

        \draw [blueedge] (yj+1)--(x1)--(yk+1)--(xk+1)--(y1)--(xj+1) (xj+2)--(y2) (yj+2)--(x2);
\end{tikzpicture}
\caption{$j<k-1$}
\label{fig-xjyk+1-b}
\end{subfigure}
\caption{The case where the edge $x_jy_{k+1}$ is red.}
\label{fig-xjyk+1}
\end{figure}

\begin{clm}\label{clm-GSk-1xkyk+1-A}
There is a blue cycle in $G$ that covers all vertices in $\Sleq{k+1}$ except $x_{k+1}$ and $y_k$.
\end{clm}

\begin{proof}
By \Cref{clm-xj+1y1x2yj+2} and \Cref{clm-xk+1y1} the edges $x_{j+1}y_1$, $x_2y_{j+2}$ and $x_1y_{k+1}$ are blue.

On the one hand if $j=k-1$, this means that the edges $x_ky_1$ and $x_2y_{k+1}$ are blue (see Figure~\ref{fig-GSk-1xkyk+1-a}). Then 
$x_2\To[E_k\setminus y_k] x_k \To y_1 \To[O_k] x_1 \To y_{k+1} \To x_2$ is a blue cycle covering $y_{k+1}$ and all vertices of $\Sleq{k}$ except $y_k$.

On the other hand ($j < k-1$), $y_{j+2}\in G_k^O$. We will also need the edge $x_jy_{k+1}$, which is blue by \Cref{clm-xjyk+1-A}. Let
$$C = x_2 \To[E_k\setminus y_k] x_{j+1} \To y_1 \To x_1 \To y_{k+1} \To x_j \To y_{j+2} \To x_2.$$ This cycle covers $G_k^E \setminus y_k$ and passes through the vertices $x_1, y_1, x_j, y_{j+2}$ in $G_k^O$. Now, we can replace the edge $x_1y_1$ of $C$ with a path that covers all remaining vertices of $G_k^O$.
\end{proof}

\begin{figure}[!ht]
\centering
\begin{subfigure}[!ht]{0.5\textwidth}
\centering
\begin{tikzpicture}[yscale=0.7, xscale=1.0]
\node at (-1.5,5) {$G_k^O$};
\fill [rounded corners, fill=blue!10] (-1,3.5) rectangle (4,6.5);

\node at (-1.5,8.5) {$G_k^E$};
\fill [rounded corners, fill=blue!10] (-1,7) rectangle (4,10);

\node [vertex, label={[label distance=0cm]180:$x_1$}] (x1) at (0,4) {};
\node at (0,5) {$\myvdots$};
\node [vertex, label={[label distance=0cm]180:$x_{k-1}$}] (xk-1) at (0,6) {};

\node [vertex, label={[label distance=0cm]180:$x_2$}] (x2) at (0,7.5) {};
\node at (0,8.5) {$\myvdots$};
\node [vertex, label={[label distance=0cm]180:$x_k$}] (xk) at (0,9.5) {};

\node [vertex, label={[label distance=0cm]180:$x_{k+1}$}] (xk+1) at (0,11) {};

\node [vertex, label={[label distance=0cm]0:$y_1$}] (y1) at (3,4) {};
\node at (3,5) {$\myvdots$};
\node [vertex, label={[label distance=0cm]0:$y_{k-1}$}] (yk-1) at (3,6) {};

\node [vertex, label={[label distance=0cm]0:$y_2$}] (y2) at (3,7.5) {};
\node at (3,8.5) {$\myvdots$};
\node [vertex, label={[label distance=0cm]0:$y_k$}] (yk) at (3,9.5) {};

\node [vertex, label={[label distance=0cm]0:$y_{k+1}$}] (yk+1) at (3,11) {};

\draw [blueedge] (x1)--(yk+1)--(x2) (y1)--(xk);

\end{tikzpicture}
\caption{$j=k-1$}
\label{fig-GSk-1xkyk+1-a}
\end{subfigure}%
\begin{subfigure}[!ht]{0.5\textwidth}
\centering
\begin{tikzpicture}[yscale=0.7, xscale=1.0]
\node at (4.5,5) {$G_k^O$};
\fill [rounded corners, fill=blue!10] (-1,3.5) rectangle (4,6.5);

\node at (4.5,8.5) {$G_k^E$};
\fill [rounded corners, fill=blue!10] (-1,7) rectangle (4,10);

\node [vertex, label={[label distance=0cm]180:$x_1$}] (x1) at (0,4) {};
\node at (0,4.5) {$\myvdots$};
\node [vertex, label={[label distance=0cm]180:$x_j$}] (xj) at (0,5) {};
\node at (0,5.5) {$\myvdots$};
\node [vertex, label={[label distance=0cm]180:$x_{k-1}$}] (xk-1) at (0,6) {};

\node [vertex, label={[label distance=0cm]180:$x_2$}] (x2) at (0,7.5) {};
\node at (0,8) {$\myvdots$};
\node [vertex, label={[label distance=0cm]180:$x_{j+1}$}] (xj+1) at (0,8.5) {};
\node at (0,9) {$\myvdots$};
\node [vertex, label={[label distance=0cm]180:$x_k$}] (xk) at (0,9.5) {};

\node [vertex, label={[label distance=0cm]180:$x_{k+1}$}] (xk+1) at (0,11) {};

\node [vertex, label={[label distance=0cm]0:$y_1$}] (y1) at (3,4) {};
\node at (3,4.5) {$\myvdots$};
\node [vertex, label={[label distance=0cm]0:$y_{j+2}$}] (yj+2) at (3,5) {};
\node at (3,5.5) {$\myvdots$};
\node [vertex, label={[label distance=0cm]0:$y_{k-1}$}] (yk-1) at (3,6) {};

\node [vertex, label={[label distance=0cm]0:$y_2$}] (y2) at (3,7.5) {};
\node at (3,8.5) {$\myvdots$};
\node [vertex, label={[label distance=0cm]0:$y_k$}] (yk) at (3,9.5) {};

\node [vertex, label={[label distance=0cm]0:$y_{k+1}$}] (yk+1) at (3,11) {};

\draw [blueedge] (x1)--(yk+1)--(xj) (xj+1)--(y1) (yj+2)--(x2);

\end{tikzpicture}
\caption{$j<k-1$}
\label{fig-GSk-1xkyk+1-b}
\end{subfigure}

\caption{A blue cycle that covers all vertices in $\Sleq{k+1}$ except $x_{k+1}$ and $y_k$.}
\label{fig-GSk-1xkyk+1}
\end{figure}

\begin{clm}\label{clm-xk+1yk+3-A}
The edge $x_{k+1}y_{k+3}$ is blue.
\end{clm}
\begin{proof}
Suppose towards a contradiction that the edge $x_{k+1}y_{k+3}$ is red (see Figure~\ref{fig-xk+1yk+3}). Hence, by \Cref{clm-GSk-1xkyk+1-A}, we see that $G$ can be vertex-partitioned into the red cycle $(x_{k+1},y_{k+3}Px_{k+3},y_{k+2})$, the edge-cycle $(x_{k+2},y_k)$ and a blue cycle that covers all vertices in $\Sleq{k+1}$ except $x_{k+1}$ and $y_k$, a contradiction. Thus, the result follows.
\begin{figure}[!ht]
\centering
\begin{tikzpicture}[yscale=0.8, xscale=1.0]
\fill [rounded corners, fill=blue!10] (-1,4.5) rectangle (4,6.5);
\fill [rounded corners, fill=blue!10] (-1,6.5) rectangle (2,7.5);
\fill [rounded corners, fill=blue!10] (1,7.5) rectangle (4,8.5);
\fill [rounded corners, fill=blue!10] (-1,6) rectangle (2,7);
\fill [rounded corners, fill=blue!10] (0,6) rectangle (3,8);
\fill [rounded corners, fill=blue!0] (-1,7.5) rectangle (1,8.5);
\fill [rounded corners, fill=blue!0] (2,6.5) rectangle (4,7.5);

\node [vertex, label={[label distance=0cm]180:$x_1$}] (x1) at (0,5) {};
\node at (0,5.5) {$\myvdots$};
\node [vertex, label={[label distance=0cm]180:$x_{k-1}$}] (xk-1) at (0,6) {};

\node [vertex, label={[label distance=0cm]180:$x_k$}] (xk) at (0,7) {};

\node [vertex, label={[label distance=0cm]180:$x_{k+1}$}] (xk+1) at (0,8) {};
\node [vertex, label={[label distance=0cm]180:$x_{k+2}$}] (xk+2) at (0,9) {};
\node [vertex, label={[label distance=0cm]180:$x_{k+3}$}] (xk+3) at (0,10) {};

\node (dots) at (1.5,11) {$\cdots$};

\node [vertex, label={[label distance=0cm]0:$y_1$}] (y1) at (3,5) {};
\node at (3,5.5) {$\myvdots$};
\node [vertex, label={[label distance=0cm]0:$y_{k-1}$}] (yk-1) at (3,6) {};

\node [vertex, label={[label distance=0cm]0:$y_k$}] (yk) at (3,7) {};

\node [vertex, label={[label distance=0cm]0:$y_{k+1}$}] (yk+1) at (3,8) {};
\node [vertex, label={[label distance=0cm]0:$y_{k+2}$}] (yk+2) at (3,9) {};
\node [vertex, label={[label distance=0cm]0:$y_{k+3}$}] (yk+3) at (3,10) {};

\draw [rededge] (xk+3)--(yk+2)--(xk+1)--(yk+3);
\draw [rededge] (xk+3) .. controls (3,11) .. (dots);
\draw [rededge] (yk+3) .. controls (0,11) .. (dots);
\draw [blueedge] (xk+2)--(yk);

\end{tikzpicture}
\caption{The case where the edge $x_{k+1}y_{k+3}$ is red.}
\label{fig-xk+1yk+3}
\end{figure}
\end{proof}

\begin{clm}\label{clm-x1ykxk+1y2}
The edges $x_1y_k$ and $x_{k+1}y_2$ are not both blue.
\end{clm}

\begin{proof}
Suppose towards a contradiction that the edges $x_1y_k$ and $x_{k+1}y_2$ are blue (see Figure~\ref{fig-x1ykxk+1y2}). In this case, we shall prove that $\Sleq{k+2}$ is a blue special set of $G$, a contradiction. The edge $x_1y_{k+1}$ is blue by \Cref{clm-xk+1y1} and $x_{j+1}y_1$ is blue by \Cref{clm-xj+1y1x2yj+2}. So,
$ x_{j+1} \To[E_k] y_k \To x_1 \To[O_k] y_1 \To x_{j+1}$ is a blue Hamiltonian cycle of $G[\Sleq{k}]$.\\And $x_{j+1} \To[E_k] y_2 \To x_{k+1} \To y_{k+1} \To x_1  \To[O_k] y_1 \To x_{j+1}$ is a blue Hamiltonian cycle of $G[\Sleq{k+1}]$. 
Together with Claim \ref{clm-GSk+2-A}, we see that $\Sleq{k+2}$ is a blue special set of $G$. Thus, the result follows.

\begin{figure}[!ht]
\centering
\begin{tikzpicture}[yscale=0.7, xscale=1.0]
\node at (-1.5,5) {$G_k^O$};
\node at (4.5,5) {$G_k^O$};
\fill [rounded corners, fill=blue!10] (-1,3.5) rectangle (4,6.5);

\node at (-1.5,8.5) {$G_k^E$};
\node at (4.5,8.5) {$G_k^E$};
\fill [rounded corners, fill=blue!10] (-1,7) rectangle (4,10);

\node [vertex, label={[label distance=0cm]180:$x_1$}] (x1) at (0,4) {};
\node at (0,5) {$\myvdots$};
\node [vertex, label={[label distance=0cm]180:$x_{k-1}$}] (xk-1) at (0,6) {};

\node [vertex, label={[label distance=0cm]180:$x_2$}] (x2) at (0,7.5) {};
\node at (0,8) {$\myvdots$};
\node [vertex, label={[label distance=0cm]180:$x_{j+1}$}] (xj+1) at (0,8.5) {};
\node at (0,9) {$\myvdots$};
\node [vertex, label={[label distance=0cm]180:$x_k$}] (xk) at (0,9.5) {};

\node [vertex, label={[label distance=0cm]180:$x_{k+1}$}] (xk+1) at (0,11) {};

\node [vertex, label={[label distance=0cm]0:$y_1$}] (y1) at (3,4) {};
\node at (3,5) {$\myvdots$};
\node [vertex, label={[label distance=0cm]0:$y_{k-1}$}] (yk-1) at (3,6) {};

\node [vertex, label={[label distance=0cm]0:$y_2$}] (y2) at (3,7.5) {};
\node at (3,8.5) {$\myvdots$};
\node [vertex, label={[label distance=0cm]0:$y_k$}] (yk) at (3,9.5) {};

\node [vertex, label={[label distance=0cm]0:$y_{k+1}$}] (yk+1) at (3,11) {};

\draw [blueedge] (yk)--(x1)--(yk+1)--(xk+1)--(y2) (xj+1)--(y1);

\end{tikzpicture}
\caption{The case where the edges $x_1y_k$ and $x_{k+1}y_2$ are blue.}
\label{fig-x1ykxk+1y2}
\end{figure}
\end{proof}

\begin{clm}\label{clm-x2yk+2}
The edge $x_2y_{k+2}$ is blue.
\end{clm}

\begin{proof}
Suppose for sake of contradiction that the edge $x_2y_{k+2}$ is red (see Figure~\ref{fig-x2yk+2}). In this case, $G$ can be vertex-partitioned into the red cycle $(x_2Py_{k+2})$, the red path $(x_1,y_2Py_k,x_{k+1})$ and the vertex $y_1$. By Claim \ref{clm-x1ykxk+1y2}, we know that the edges $x_1y_k$ and $x_{k+1}y_2$ are not both blue. Hence, we see that $G$ can be vertex-partitioned into $3$ monochromatic cycles (one of which is an edge-cycle), a contradiction. Thus, the result follows.

\begin{figure}[!ht]
\centering
\begin{tikzpicture}[scale=1.0]
\node [vertex, label={[label distance=0cm]270:$x_1$}] (x1) at (-1,0) {};
\node [vertex, label={[label distance=0cm]270:$y_2$}] (y2) at (0,0) {};
\node (dotsx) at (1,0) {$\cdots$};
\node [vertex, label={[label distance=0cm]270:$y_k$}] (yk) at (2,0) {};
\node [vertex, label={[label distance=0cm]270:$x_{k+1}$}] (xk+1) at (3,0) {};
\node [vertex, label={[label distance=0cm]270:$y_{k+2}$}] (yk+2) at (4,0) {};
\node (dotsn) at (6.5,0) {$\cdots$};
\node [vertex, label={[label distance=0cm]270:$x_2$}] (x2) at (9,0) {};
\node [vertex, label={[label distance=0cm]270:$y_1$}] (y1) at (10,0) {};

\draw [rededge] (x1)--(y2)--(dotsx)--(yk)--(xk+1)--(yk+2)--(dotsn)--(x2)--(y1);
\draw [rededge] (yk+2) to [out=30, in=150] (x2);
\end{tikzpicture}
\caption{The case where the edge $x_2y_{k+2}$ is red.}
\label{fig-x2yk+2}
\end{figure}
\end{proof}

\begin{clm}\label{clm-xk+3yk+1-A}
The edge $x_{k+3}y_{k+1}$ is blue.
\end{clm}
\begin{proof}
Suppose towards a contradiction that the edge $x_{k+3}y_{k+1}$ is red (see Figure~\ref{fig-xk+3yk+1}). By \Cref{clm-xj+1y1x2yj+2} and \Cref{clm-x2yk+2} the edges $x_{j+1}y_1$ and $x_2y_{k+2}$ are blue. Let 
$$C = x_{j+1}\To[E_k\setminus y_k] x_2 \To y_{k+2} \To x_j \To[O_k] y_1 \To x_{j+1}.$$ $C$ is a blue cycle that covers $(\Sleq{k} \setminus y_k) \cup \{y_{k+2}\}$. Then $G$ can be vertex-partitioned into the red cycle $(x_{k+2},y_{k+3}Px_{k+3},y_{k+1})$, the edge-cycle $(x_{k+1},y_k)$ and $C$, a contradiction. Thus, the result follows.
\end{proof}

\begin{figure}[!ht]
\centering
        \begin{tikzpicture}[yscale=0.7, xscale=1.0]
        \node at (-1.5,5) {$G_k^O$};
        \node at (4.5,5) {$G_k^O$};
        \fill [rounded corners, fill=blue!10] (-1,3.5) rectangle (4,6.5);

        \node at (-1.5,8.5) {$G_k^E$};
        \node at (4.5,8.5) {$G_k^E$};
        \fill [rounded corners, fill=blue!10] (-1,7) rectangle (4,10);

        \node [vertex, label={[label distance=0cm]180:$x_1$}] (x1) at (0,4) {};
        \node at (0,4.5) {$\myvdots$};
        \node [vertex, label={[label distance=0cm]180:$x_j$}] (xj) at (0,5) {};
        \node at (0,5.5) {$\myvdots$};
        \node [vertex, label={[label distance=0cm]180:$x_{k-1}$}] (xk-1) at (0,6) {};

        \node [vertex, label={[label distance=0cm]180:$x_2$}] (x2) at (0,7.5) {};
        \node at (0,8) {$\myvdots$};
        \node [vertex, label={[label distance=0cm]180:$x_{j+1}$}] (xj+1) at (0,8.5) {};
        \node at (0,9) {$\myvdots$};
        \node [vertex, label={[label distance=0cm]180:$x_k$}] (xk) at (0,9.5) {};

        \node [vertex, label={[label distance=0cm]180:$x_{k+1}$}] (xk+1) at (0,11) {};
        \node [vertex, label={[label distance=0cm]180:$x_{k+2}$}] (xk+2) at (0,12) {};
        \node [vertex, label={[label distance=0cm]180:$x_{k+3}$}] (xk+3) at (0,13) {};

        \node (dots) at (1.5,14) {$\cdots$};

        \node [vertex, label={[label distance=0cm]0:$y_1$}] (y1) at (3,4) {};
        \node at (3,5) {$\myvdots$};
        \node [vertex, label={[label distance=0cm]0:$y_{k-1}$}] (yk-1) at (3,6) {};

        \node [vertex, label={[label distance=0cm]0:$y_2$}] (y2) at (3,7.5) {};
        \node at (3,8.5) {$\myvdots$};
        \node [vertex, label={[label distance=0cm]0:$y_k$}] (yk) at (3,9.5) {};

        \node [vertex, label={[label distance=0cm]0:$y_{k+1}$}] (yk+1) at (3,11) {};
        \node [vertex, label={[label distance=0cm]0:$y_{k+2}$}] (yk+2) at (3,12) {};
        \node [vertex, label={[label distance=0cm]0:$y_{k+3}$}] (yk+3) at (3,13) {};

        \draw [rededge] (xk+3)--(yk+1)--(xk+2)--(yk+3) (xk+1)--(yk);
        \draw [rededge] (xk+3) .. controls (3,14) .. (dots);
        \draw [rededge] (yk+3) .. controls (0,14) .. (dots);
        \draw [blueedge] (xj)--(yk+2)--(x2) (y1)--(xj+1);
\end{tikzpicture}
\caption{The case where the edge $x_{k+3}y_{k+1}$ is red.}
\label{fig-xk+3yk+1}
\end{figure}
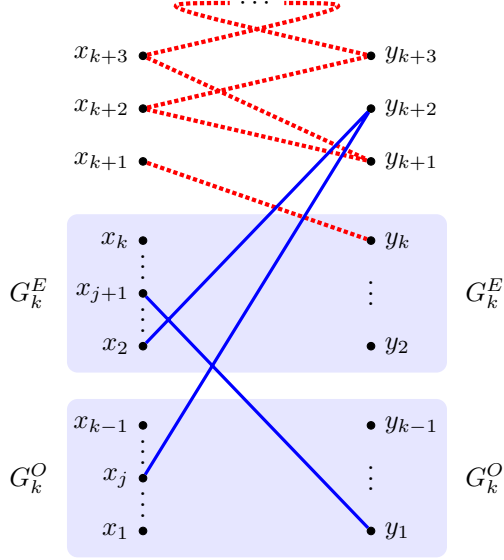

\begin{clm}\label{clm-GSk+3-A}
The subgraph $G[\Sleq{k+3}]$ has a blue Hamiltonian cycle that uses the edges $x_{k+2}y_k$ and $x_{k+3}y_{k+3}$. Consequently, the edge $x_{k+4}y_{k+4}$ is blue.
\end{clm}

\begin{proof}
By \Cref{clm-GSk+2-A}, start with a blue Hamiltonian cycle $C$ of $G[\Sleq{k+2}]$ that uses the edges $x_{k+1}y_{k+1}$ and $x_{k+2}y_k$. \Cref{clm-GSk+2-A} also tells us that $x_{k+3}y_{k+3}$ is blue. Now, extend $C$ by replacing the edge $x_{k+1}y_{k+1}$ by the path $x_{k+1}y_{k+3}x_{k+3}y_{k+1}$ (using Claims \ref{clm-xk+1yk+3-A} and \ref{clm-xk+3yk+1-A}). This gives the desired blue Hamiltonian cycle of $G[\Sleq{k+3}]$. By \Cref{rmk-xiyixiyi+2}(a), it follows that the edge $x_{k+4}y_{k+4}$ is blue.
\end{proof}

\begin{clm}\label{clm-xk+4ykxk+2yk+4-A}
The edges $x_{k+4}y_k$ and $x_{k+2}y_{k+4}$ are blue.
\end{clm}

\begin{proof}
Suppose for sake of contradiction that the edge $x_{k+4}y_k$ is red (see Figure~\ref{fig-xk+4yk}). Hence, by Claim \ref{clm-GSk-1xkyk+1-A}, we see that $G$ can be vertex-partitioned into the red cycle $(x_{k+1},y_{k+2},x_{k+3},y_{k+4}Px_{k+4},y_k)$, the edge-cycle $(x_{k+2},y_{k+3})$ and a blue cycle that covers all vertices in $\Sleq{k+1}$ except $x_{k+1}$ and $y_k$, a contradiction. Thus, the edge $x_{k+4}y_k$ is blue.

\begin{figure}[!ht]
\centering
\begin{subfigure}[!ht]{0.4\textwidth}%
\centering
\begin{tikzpicture}[yscale=0.8, xscale=1.0]
        \fill [rounded corners, fill=blue!10] (-1,4.5) rectangle (4,6.5);
        \fill [rounded corners, fill=blue!10] (-1,6.5) rectangle (2,7.5);
        \fill [rounded corners, fill=blue!10] (1,7.5) rectangle (4,8.5);
        \fill [rounded corners, fill=blue!10] (-1,6) rectangle (2,7);
        \fill [rounded corners, fill=blue!10] (0,6) rectangle (3,8);
        \fill [rounded corners, fill=blue!0] (-1,7.5) rectangle (1,8.5);
        \fill [rounded corners, fill=blue!0] (2,6.5) rectangle (4,7.5);

        \node [vertex, label={[label distance=0cm]180:$x_1$}] (x1) at (0,5) {};
        \node at (0,5.5) {$\myvdots$};
        \node [vertex, label={[label distance=0cm]180:$x_{k-1}$}] (xk-1) at (0,6) {};

        \node [vertex, label={[label distance=0cm]180:$x_k$}] (xk) at (0,7) {};

        \node [vertex, label={[label distance=0cm]180:$x_{k+1}$}] (xk+1) at (0,8) {};
        \node [vertex, label={[label distance=0cm]180:$x_{k+2}$}] (xk+2) at (0,9) {};
        \node [vertex, label={[label distance=0cm]180:$x_{k+3}$}] (xk+3) at (0,10) {};
        \node [vertex, label={[label distance=0cm]180:$x_{k+4}$}] (xk+4) at (0,11) {};

        \node (dots) at (1.5,12) {$\cdots$};

        \node [vertex, label={[label distance=0cm]0:$y_1$}] (y1) at (3,5) {};
        \node at (3,5.5) {$\myvdots$};
        \node [vertex, label={[label distance=0cm]0:$y_{k-1}$}] (yk-1) at (3,6) {};

        \node [vertex, label={[label distance=0cm]0:$y_k$}] (yk) at (3,7) {};

        \node [vertex, label={[label distance=0cm]0:$y_{k+1}$}] (yk+1) at (3,8) {};
        \node [vertex, label={[label distance=0cm]0:$y_{k+2}$}] (yk+2) at (3,9) {};
        \node [vertex, label={[label distance=0cm]0:$y_{k+3}$}] (yk+3) at (3,10) {};
        \node [vertex, label={[label distance=0cm]0:$y_{k+4}$}] (yk+4) at (3,11) {};

        \draw [rededge] (xk+4)--(yk)--(xk+1)--(yk+2)--(xk+3)--(yk+4) (xk+2)--(yk+3);
        \draw [rededge] (xk+4) .. controls (3,12) .. (dots);
        \draw [rededge] (yk+4) .. controls (0,12) .. (dots);
\end{tikzpicture}
\caption{When $x_{k+4}y_k$ is red.}
\label{fig-xk+4yk}
\end{subfigure}%
\begin{subfigure}[!ht]{0.59\textwidth}
\centering
\begin{tikzpicture}[yscale=0.7, xscale=1.0]
\node at (-1.5,5) {$G_k^O$};
\node at (4.5,5) {$G_k^O$};
\fill [rounded corners, fill=blue!10] (-1,3.5) rectangle (4,6.5);

\node at (-1.5,8.5) {$G_k^E$};
\node at (4.5,8.5) {$G_k^E$};
\fill [rounded corners, fill=blue!10] (-1,7) rectangle (4,10);

\node [vertex, label={[label distance=0cm]180:$x_1$}] (x1) at (0,4) {};
\node at (0,4.5) {$\myvdots$};
\node [vertex, label={[label distance=0cm]180:$x_j$}] (xj) at (0,5) {};
\node at (0,5.5) {$\myvdots$};
\node [vertex, label={[label distance=0cm]180:$x_{k-1}$}] (xk-1) at (0,6) {};

\node [vertex, label={[label distance=0cm]180:$x_2$}] (x2) at (0,7.5) {};
\node at (0,8) {$\myvdots$};
\node [vertex, label={[label distance=0cm]180:$x_{j+1}$}] (xj+1) at (0,8.5) {};
\node at (0,9) {$\myvdots$};
\node [vertex, label={[label distance=0cm]180:$x_k$}] (xk) at (0,9.5) {};

\node [vertex, label={[label distance=0cm]180:$x_{k+1}$}] (xk+1) at (0,11) {};
\node [vertex, label={[label distance=0cm]180:$x_{k+2}$}] (xk+2) at (0,12) {};
\node [vertex, label={[label distance=0cm]180:$x_{k+3}$}] (xk+3) at (0,13) {};
\node [vertex, label={[label distance=0cm]180:$x_{k+4}$}] (xk+4) at (0,14) {};

\node (dots) at (1.5,15) {$\cdots$};

\node [vertex, label={[label distance=0cm]0:$y_1$}] (y1) at (3,4) {};
\node at (3,5) {$\myvdots$};
\node [vertex, label={[label distance=0cm]0:$y_{k-1}$}] (yk-1) at (3,6) {};

\node [vertex, label={[label distance=0cm]0:$y_2$}] (y2) at (3,7.5) {};
\node at (3,8.5) {$\myvdots$};
\node [vertex, label={[label distance=0cm]0:$y_k$}] (yk) at (3,9.5) {};

\node [vertex, label={[label distance=0cm]0:$y_{k+1}$}] (yk+1) at (3,11) {};
\node [vertex, label={[label distance=0cm]0:$y_{k+2}$}] (yk+2) at (3,12) {};
\node [vertex, label={[label distance=0cm]0:$y_{k+3}$}] (yk+3) at (3,13) {};
\node [vertex, label={[label distance=0cm]0:$y_{k+4}$}] (yk+4) at (3,14) {};

\draw [rededge] (xk+4)--(yk+3)--(xk+2)--(yk+4);
\draw [rededge] (xk+4) .. controls (3,15) .. (dots);
\draw [rededge] (yk+4) .. controls (0,15) .. (dots);
\draw [blueedge] (xj)--(yk+2)--(x2) (x1)--(yk+1)--(xk+1)--(y1)--(xj+1);
\end{tikzpicture}
\caption{When $x_{k+2}y_{k+4}$ is red.}
\label{fig-xk+2yk+4}
\end{subfigure}%
\caption{The cases of \Cref{clm-xk+4ykxk+2yk+4-A}.}
\end{figure}

Now, suppose towards a contradiction that the edge $x_{k+2}y_{k+4}$ is red (see Figure~\ref{fig-xk+2yk+4}). By \Cref{clm-xj+1y1x2yj+2} and \Cref{clm-x2yk+2} the edges $x_{j+1}y_1$ and $x_2y_{k+2}$ are blue. Together with \Cref{clm-xk+1y1} we build the blue cycle
$$C = x_{j+1} \To[E_k\setminus y_k] x_2 \To y_{k+2}\To x_j \To[O_k\setminus y_1] x_1 \To y_{k+1} \To x_{k+1} \To y_1 \To x_{j+1}.$$
$C$ covers $(\Sleq{k+1}\setminus y_k) \cup \{y_{k+2}\} $.
Thus $G$ can be partitioned into the red cycle $(x_{k+2},y_{k+4}Px_{k+4},y_{k+3})$, the edge-cycle $(x_{k+3},y_k)$ and $C$. Hence, we get a contradiction. Thus, the edge $x_{k+2}y_{k+4}$ is blue.
\end{proof}

\begin{clm}\label{clm-GSk+4-A}
The subgraph $G[\Sleq{k+4}]$ has a blue Hamiltonian cycle that uses the edges $x_{k+3}y_{k+3}$ and $x_{k+4}y_{k+4}$.
\end{clm}

\begin{proof}
Using \Cref{clm-GSk+3-A}, take a blue Hamiltonian cycle of $G[\Sleq{k+3}]$ that uses the edges $x_{k+2}y_k$ and $x_{k+3}y_{k+3}$ and extend it to a blue Hamiltonian cycle of $G[\Sleq{k+4}]$ that uses the edges $x_{k+3}y_{k+3}$ and $x_{k+4}y_{k+4}$ by replacing the edge $x_{k+2}y_k$ by the path $(x_{k+2},y_{k+4},x_{k+4},y_k)$, which is blue by Claims~\ref{clm-GSk+3-A} and \ref{clm-xk+4ykxk+2yk+4-A}.
\end{proof}

Thus, by Claims \ref{clm-GSk+2-A}, \ref{clm-GSk+3-A} and \ref{clm-GSk+4-A}, we see that $\Sleq{k+4}$ is a blue special set of $G$, which is our final contradiction for Case A.

\begin{center}
\textbf{Case B:} $k$ is odd and $k\ge 4$, therefore $j$ is even with $2 \leq j < k \leq n-2$.
\end{center}

Most proofs are similar to the ones in Case A. We will give the details on how to derive the A proofs to B proofs. Of course, we should also replace any reference to a Claim A to its Claim B counterpart.
\stepcounter{case}

\begin{clm}\label{clm-x2ykxk-1y1}\mbox{}
\begin{enumerate}[(i)] 
         \item \label{clm-x2ykxk-1y1-itemi} If there exists $1 \leq l \leq k$ odd such that $x_{k+1}y_l$ is red, then $x_2y_k$ and $x_{k-1}y_1$ are not both blue.
         \item \label{clm-x2ykxk-1y1-itemii} If there exists $1 \leq l \leq k$ odd such that $x_{k+1}y_l$ is red, then $x_ky_2$ and $x_1y_{k-1}$ are not both blue.
\end{enumerate}
\end{clm}

\begin{proof}
Interchange the roles of $x_1$ and $x_2$, of $y_1$ and $y_2$ and of $G_k^O$ and $G_k^E$ in the proof of Claim \ref{clm-x1ykxk-1y2}.
\end{proof}

\begin{clm}\label{clm-xk+1y2}
The edges $x_{k+1}y_2$ and $x_2y_{k+1}$ are blue.
\end{clm}

\begin{proof}
Suppose towards a contradiction that the edge $x_{k+1}y_2$ is red (see Figure~\ref{fig-xk+1y2}). In this case, $G$ can be vertex-partitioned into the red cycle $(x_{k+1}Py_2)$, the red path $(y_1,x_2Px_{k-1},y_k)$ and the vertex~$x_1$ (a difference here is that in \Cref{clm-xk+1y1} we did not have this hanging vertex $x_1$).  By Claim \ref{clm-x2ykxk-1y1}(\ref{clm-x2ykxk-1y1-itemi}), we know that the edges $x_2y_k$ and $x_{k-1}y_1$ are not both blue. Hence, we see that $G$ can be vertex-partitioned into $3$ monochromatic cycles, a contradiction. Thus, the edge $x_{k+1}y_2$ is blue. Analogously, swapping the roles of $X$ and $Y$ and using \Cref{clm-x2ykxk-1y1}(\ref{clm-x2ykxk-1y1-itemii}), the edge $x_2y_{k+1}$ is blue too. Thus, the result follows.

\begin{figure}[!ht]
\centering
\begin{tikzpicture}[scale=1.0]
\node [vertex, label={[label distance=0cm]270:$x_1$}] (x1) at (-1,0) {};
\node [vertex, label={[label distance=0cm]270:$y_2$}] (y2) at (0,0) {};
\node (dotsn) at (2.5,0) {$\cdots$};
\node [vertex, label={[label distance=0cm]270:$x_{k+1}$}] (xk+1) at (5,0) {};
\node [vertex, label={[label distance=0cm]270:$y_k$}] (yk) at (6,0) {};
\node [vertex, label={[label distance=0cm]270:$x_{k-1}$}] (xk-1) at (7,0) {};
\node (dotsy) at (8,0) {$\cdots$};
\node [vertex, label={[label distance=0cm]270:$x_2$}] (x2) at (9,0) {};
\node [vertex, label={[label distance=0cm]270:$y_1$}] (y1) at (10,0) {};

\draw [rededge] (x1)--(y2)--(dotsn)--(xk+1)--(yk)--(xk-1)--(dotsy)--(x2)--(y1);
\draw [rededge] (xk+1) to [out=150, in=30] (y2);
\end{tikzpicture}
\caption{The case where the edge $x_{k+1}y_2$ is red.}
\label{fig-xk+1y2}
\end{figure}
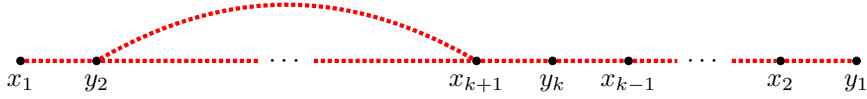
\end{proof}

\vspace{0.3cm}
\begin{clm}\label{clm-xk+2yk+2-B}
The edge $x_{k+2}y_{k+2}$ is blue.
\end{clm}

\begin{proof}
Just interchange the roles of $x_1$ and $x_2$, of $y_1$ and $y_2$ and of $G_k^O$ and $G_k^E$.
\end{proof}

As in Case A, from this point, without loss of generality
        \textbf{we assume  that the edge $x_{k+1}y_j$ is red.}
By \Cref{clm-xk+1y2} we know that $j \neq 2$. And as in Case A, if $k = n-2$, then the proof ends here: since $G$ is a red zigzag graph, $x_ny_n$ is red, a contradiction to \Cref{clm-xk+2yk+2-B}. Hence from now:  \[
        2 < j \leq k \le n-3
\]

\begin{clm}\label{clm-xj+1y2x1yj+2}
The edges $x_{j+1}y_2$ and $x_1y_{j+2}$ are blue. In particular, when $j=k-1$, this means that $x_ky_2$ and $x_1y_{k+1}$ are blue.
\end{clm}

\begin{proof}
This proof is similar to \Cref{clm-xj+1y1x2yj+2}, with minor changes. First, to prove that the edge  $x_{j+1}y_2$ is blue, we do as in \Cref{clm-xk+1y2}, interchanging the roles of $x_1$ and $x_2$, and of $y_1$ and $y_2$.

Now, suppose for sake of contradiction that the edge $x_1y_{j+2}$ is red (see Figure~\ref{fig-x1yj+2}). Then, $G$ can be vertex-partitioned into the red cycle $(x_{k+1}Py_{j+2},x_1Py_j)$, the red path $(y_1,x_2Px_{k-1},y_k)$ and the vertex $x_{j+1}$. By Claim \ref{clm-x2ykxk-1y1}, we know that the edges $x_2y_k$ and $x_{k-1}y_1$ are not both blue. Hence, $G$ can be vertex-partitioned into $3$ monochromatic cycles, a contradiction. Thus, the edge $x_1y_{j+2}$ is blue.

\begin{figure}[!ht]
\centering
\begin{tikzpicture}[scale=0.96]
\node [vertex, label={[label distance=0cm]270:$x_1$}] (x1) at (0,0) {};
\node (dotsx) at (1,0) {$\cdots$};
\node [vertex, label={[label distance=0cm]270:$y_j$}] (yj) at (2,0) {};
\node [vertex, label={[label distance=0cm]270:$x_{j+1}$}] (xj+1) at (3,0) {};
\node [vertex, label={[label distance=0cm]270:$y_{j+2}$}] (yj+2) at (4,0) {};
\node (dotsn) at (5.5,0) {$\cdots$};
\node [vertex, label={[label distance=0cm]270:$x_{k+1}$}] (xk+1) at (7,0) {};
\node [vertex, label={[label distance=0cm]270:$y_k$}] (yk) at (8,0) {};
\node [vertex, label={[label distance=0cm]270:$x_{k-1}$}] (xk-1) at (9,0) {};
\node (dotsy) at (10,0) {$\cdots$};
\node [vertex, label={[label distance=0cm]270:$x_2$}] (x2) at (11,0) {};
\node [vertex, label={[label distance=0cm]270:$y_1$}] (y1) at (12,0) {};

\draw [rededge] (x1)--(dotsx)--(yj)--(xj+1)--(yj+2)--(dotsn)--(xk+1)--(yk)--(xk-1)--(dotsy)--(x2)--(y1);
\draw [rededge] (yj+2) to [out=150, in=30] (x1);
\draw [rededge] (xk+1) to [out=150, in=30] (yj);

\end{tikzpicture}
\caption{The case where the edge $x_1y_{j+2}$ is red.}
\label{fig-x1yj+2}
\end{figure}
\end{proof}

\begin{clm}\label{clm-xjyk+2xk+2yk-B}
The edges $x_jy_{k+2}$ and $x_{k+2}y_k$ are blue.
\end{clm}

\begin{proof}
To show that $x_jy_{k+2}$ is blue, use the proof of \Cref{clm-xjyk+2xk+2yk-A} as is (see Figure~\ref{fig-xjyk+2}). Then to show that $x_{k+2}y_k$ is blue, do as in Claim~\ref{clm-xjyk+2xk+2yk-A}, interchanging the roles of $x_1$ and $x_2$, of $y_1$ and $y_2$, of $O_k$ and $E_k$, and of $G_k^O$ and $G_k^E$.
\end{proof}

\begin{clm}\label{clm-GSk+2-B}
The subgraph $G[\Sleq{k+2}]$ has a blue Hamiltonian cycle that uses the edges $x_{k+1}y_{k+1}$, $x_{k+2}y_{k+2}$ and $x_{k+2}y_k$. Furthermore, the edge $x_{k+3}y_{k+3}$ is blue.
\end{clm}

\begin{proof}
Interchange the roles of $x_1$ and $x_2$, of $y_1$ and $y_2$ and of $G_k^O$ and $G_k^E$ in the proof of \Cref{clm-GSk+2-A}.
\end{proof}

As in Case A, if $k = n-3$ then the proof ends here: $G$ is a red zigzag graph, so $x_ny_n$ is red, a contradiction to \Cref{clm-GSk+2-B}. Hence, in the remaining claims we have
\[
        k \leq n-4.
\]

\begin{clm}\label{clm-xjyk+1-B}
The edge $x_jy_{k+1}$ is blue.
\end{clm}

\begin{proof}
Interchange the roles of $x_1$ and $x_2$, of $y_1$ and $y_2$ and of $G_k^O$ and $G_k^E$  in the proof of \Cref{clm-xjyk+1-A}.
\end{proof}

\begin{clm}\label{clm-GSk-1xkyk+1-B}
There is a blue cycle in $G$ that covers all vertices in $\Sleq{k+1}$ except $x_{k+1}$ and $y_k$.
\end{clm}

\begin{proof}
Interchange the roles of $x_1$ and $x_2$, of $y_1$ and $y_2$ and of $G_k^O$ and $G_k^E$ in the proof of \Cref{clm-GSk-1xkyk+1-A}.
\end{proof}

\begin{clm}\label{clm-xk+1yk+3-B}
The edge $x_{k+1}y_{k+3}$ is blue.
\end{clm}

\begin{proof}
The proof of Claim \ref{clm-xk+1yk+3-A} can be used as is.
\end{proof}

\begin{clm}\label{clm-x2ykxk+1y1}
The edges $x_2y_k$ and $x_{k+1}y_1$ are not both blue.
\end{clm}

\begin{proof}
Interchange the roles of $x_1$ and $x_2$, of $y_1$ and $y_2$ and of $G_k^O$ and $G_k^E$  in the proof of \Cref{clm-x1ykxk+1y2}.
\end{proof}

\begin{clm}\label{clm-x1yk+2}
The edge $x_1y_{k+2}$ is blue.
\end{clm}

\begin{proof}
Suppose towards a contradiction that the edge $x_1y_{k+2}$ is red (see Figure~\ref{fig-x1yk+2}). In this case, $G$ can be vertex-partitioned into the red cycle $(x_1Py_{k+2})$ and the red path $(x_{k+1},y_kPx_2,y_1)$. By Claim \ref{clm-x2ykxk+1y1}, we know that the edges $x_2y_k$ and $x_{k+1}y_1$ are not both blue. Hence, we see that $G$ can be vertex-partitioned into at most $3$ monochromatic cycles, a contradiction. 

\begin{figure}[!ht]
\centering
\begin{tikzpicture}[scale=1.0]
\node [vertex, label={[label distance=0cm]270:$x_1$}] (x1) at (-1,0) {};
\node (dotsn) at (2,0) {$\cdots$};
\node [vertex, label={[label distance=0cm]270:$y_{k+2}$}] (yk+2) at (5,0) {};
\node [vertex, label={[label distance=0cm]270:$x_{k+1}$}] (xk+1) at (6,0) {};
\node [vertex, label={[label distance=0cm]270:$y_k$}] (yk) at (7,0) {};
\node (dotsy) at (8,0) {$\cdots$};
\node [vertex, label={[label distance=0cm]270:$x_2$}] (x2) at (9,0) {};
\node [vertex, label={[label distance=0cm]270:$y_1$}] (y1) at (10,0) {};

\draw [rededge] (x1)--(dotsn)--(yk+2)--(xk+1)--(yk)--(dotsy)--(x2)--(y1);
\draw [rededge] (yk+2) to [out=150, in=30] (x1);

\end{tikzpicture}
\caption{The case where the edge $x_1y_{k+2}$ is red.}
\label{fig-x1yk+2}
\end{figure}
\end{proof}

\begin{clm}\label{clm-xk+3yk+1-B}
The edge $x_{k+3}y_{k+1}$ is blue.
\end{clm}

\begin{proof}
Interchange the roles of $x_1$ and $x_2$, of $y_1$ and $y_2$ and of $G_k^O$ and $G_k^E$ in the proof of Claim \ref{clm-xk+3yk+1-A}.
\end{proof}

\begin{clm}\label{clm-GSk+3-B}
The subgraph $G[\Sleq{k+3}]$ has a blue Hamiltonian cycle that uses the edges $x_{k+2}y_k$ and $x_{k+3}y_{k+3}$. Consequently, the edge $x_{k+4}y_{k+4}$ is blue.
\end{clm}

\begin{proof}
The proof of Claim \ref{clm-GSk+3-A} can almost be used as is: just replace Claim \ref{clm-GSk+2-A} by \ref{clm-GSk+2-B}.
\end{proof}

\begin{clm}\label{clm-xk+4ykxk+2yk+4-B}
The edges $x_{k+4}y_k$ and $x_{k+2}y_{k+4}$ are blue.
\end{clm}

\begin{proof}
Interchange the roles of $x_1$ and $x_2$, of $y_1$ and $y_2$ and of $G_k^O$ and $G_k^E$ in the proof of Claim \ref{clm-xk+4ykxk+2yk+4-A}.
\end{proof}

\begin{clm}\label{clm-GSk+4-B}
The subgraph $G[\Sleq{k+4}]$ has a blue Hamiltonian cycle that uses the edges $x_{k+3}y_{k+3}$ and $x_{k+4}y_{k+4}$.
\end{clm}

\begin{proof}

We can use the proof of Claim \ref{clm-GSk+4-A} almost as is: just use \Cref{clm-GSk+3-B} instead of \Cref{clm-GSk+3-A}.
\end{proof}

Thus, by Claims \ref{clm-GSk+2-B}, \ref{clm-GSk+3-B} and \ref{clm-GSk+4-B}, we see that $\Sleq{k+4}$ is a blue special set of $G$, which is our final contradiction for Case B.

This ends the proof of \Cref{lmm-Knn-withpath-evenplait}.
\end{proof}

Now, we may finally prove Lemma \ref{lmm-Knnwithpath-3cycles}.

\begin{proof}[Proof of \Cref{lmm-Knnwithpath-3cycles}]
Let us take $G$ as in the statement: a balanced complete bipartite graph whose edges are coloured red or blue, and which has a monochromatic Hamiltonian path. Without loss of generality, we can assume this path is red and that its vertices are labelled so that it is a red zigzag graph. Suppose towards a contradiction that $G$ cannot be partitioned into three monochromatic cycles. Let the number of vertices of $G$ be $2n$.

By Remark \ref{rmk-xiyixiyi+2}(a), we may assume that the edges $x_1y_1$ and $x_2y_2$ are blue, since otherwise we would be done. Thus, the subgraph $G[\Sleq{2}]$ is a blue even plait and so we may apply Lemma~\ref{lmm-Knn-withpath-evenplait} iteratively until $k = n-2$ inclusive. Therefore $\Sleq{n-1}$ is a blue even plait, hence it can be decomposed into two blue cycles $C_E$ and $C_O$ by \Cref{rmk-evenplait-2cycles}. Finally, $G$ can be decomposed into the red edge-cycle $x_ny_n$ and the two blue cycles $C_E$ and $C_O$, which is a contradiction and concludes the proof.
\end{proof}

\section{Conclusions and perspectives}
In this paper, we proved a new bound on the number of monochromatic cycles needed to partition the vertices of any 2-edge-coloured complete balanced bipartite graph. We made quite an improvement as the previous known bound was 12 and we reduced it to 4. Also we can point out that our work concerns all graphs, while many papers tackling similar problems focus on sufficiently large graphs.
Working on all graphs makes it possible to not resort to strong results (such as the regularity lemma). Our work is self-contained apart from Stein's result (\Cref{cor-Knn-pathcycle}) for which we also had an alternate but very similar proof, hence not worth being shown here.

We wrote the proof of our main result in a very detailed way, with the goal to make it easy to be verified. So it ended up being somewhat long. Of course, one can make it shorter (but harder to check) by omitting some trivial details and remarks or not explicitly listing all cycles (the reader is free to skip these details and search for the cycles in the pictures). We could also have omitted the ``B claims'' (when $k$ is odd), asking the reader to believe that the proofs are similar to the ones of the case when $k$ is even. However, we did not find a way to make the analysis intrinsically shorter. A good reason for this is that many of those cases have to be checked even to deal with small values of $k$ (say $k\le 10$).

We also showed that if the colouring is split, then we can partition the vertices into at most three monochromatic cycles and described exactly when 2 cycles are enough (and when they are not). Besides, we gave examples of a $2$-colouring of a complete balanced bipartite graph that is not split and also cannot be partitioned into 2 cycles. Therefore, it remains to decide if the 2-colour cycle partition number of balanced bipartite graphs is 3 or 4.


\bibliographystyle{acm}
\bibliography{Bibliography}

\begin{thebibliography}{10}

\bibitem{alln08}
{\sc Allen, P.}
\newblock Covering two-edge-coloured complete graphs with two disjoint
  monochromatic cycles.
\newblock {\em Combinatorics, Probability and Computing 17}, 04 (2008),
  471--486.

\bibitem{balogh2014partitioning}
{\sc Balogh, J., Bar{\'a}t, J., Gerbner, D., Gy{\'a}rf{\'a}s, A., and
  S{\'a}rk{\"o}zy, G.~N.}
\newblock Partitioning 2-edge-colored graphs by monochromatic paths and cycles.
\newblock {\em Combinatorica 34}, 5 (2014), 507--526.

\bibitem{bss-tmss10}
{\sc Bessy, S., and Thomass{\'e}, S.}
\newblock Partitioning a graph into a cycle and an anticycle, a proof of
  lehel's conjecture.
\newblock {\em Journal of Combinatorial Theory, Series B 100}, 2 (2010),
  176--180.

\bibitem{cnln-stn16}
{\sc Conlon, D., and Stein, M.}
\newblock Monochromatic cycle partitions in local edge colorings.
\newblock {\em Journal of Graph Theory 81}, 2 (2016), 134--145.

\bibitem{debiasio2017monochromatic}
{\sc DeBiasio, L., and Nelsen, L.~L.}
\newblock Monochromatic cycle partitions of graphs with large minimum degree.
\newblock {\em Journal of Combinatorial Theory, Series B 122\/} (2017),
  634--667.

\bibitem{erds-grfs-pbr91}
{\sc Erd{\H{o}}s, P., Gy{\'a}rf{\'a}s, A., and Pyber, L.}
\newblock Vertex coverings by monochromatic cycles and trees.
\newblock {\em Journal of Combinatorial Theory, Series B 51}, 1 (1991), 90--95.

\bibitem{grcsr-grfs67}
{\sc Gerencs{\'e}r, L., and Gy{\'a}rf{\'a}s, A.}
\newblock On ramsey-type problems.
\newblock {\em Ann. Univ. Sci. Budapest. E{\"o}tv{\"o}s Sect. Math 10\/}
  (1967), 167--170.

\bibitem{grfs83}
{\sc Gy{\'a}rf{\'a}s, A.}
\newblock Vertex coverings by monochromatic paths and cycles.
\newblock {\em Journal of Graph Theory 7}, 1 (1983), 131--135.

\bibitem{grfs16}
{\sc Gy{\'a}rf{\'a}s, A.}
\newblock Vertex covers by monochromatic pieces — a survey of results and
  problems.
\newblock {\em Discrete Mathematics 339}, 7 (2016), 1970--1977.

\bibitem{grfs-lhl73}
{\sc Gy{\'a}rf{\'a}s, A., and Lehel, J.}
\newblock A ramsey-type problem in directed and bipartite graphs.
\newblock {\em Periodica Mathematica Hungarica 3}, 3-4 (1973), 299--304.

\bibitem{grfs-rsznk-skz-szmrd06.1}
{\sc Gy{\'a}rf{\'a}s, A., Ruszink{\'o}, M., S{\'a}rk{\"o}zy, G.~N., and
  Szemer{\'e}di, E.}
\newblock An improved bound for the monochromatic cycle partition number.
\newblock {\em Journal of Combinatorial Theory, Series B 96}, 6 (2006),
  855--873.

\bibitem{grfs-rsznk-skz-szmrd11}
{\sc Gy{\'a}rf{\'a}s, A., Ruszink{\'o}, M., S{\'a}rk{\"o}zy, G.~N., and
  Szemer{\'e}di, E.}
\newblock Partitioning 3-colored complete graphs into three monochromatic
  cycles.
\newblock {\em Electronic J. of Combinatorics 18}, 1 (2011).

\bibitem{gyarfas2013monochromatic}
{\sc Gy{\'a}rf{\'a}s, A., and S{\'a}rk{\"o}zy, G.~N.}
\newblock Monochromatic path and cycle partitions in hypergraphs.
\newblock {\em the electronic journal of combinatorics\/} (2013), P18--P18.

\bibitem{haxell1997partitioning}
{\sc Haxell, P.~E.}
\newblock Partitioning complete bipartite graphs by monochromatic cycles.
\newblock {\em Journal of Combinatorial Theory, Series B 69}, 2 (1997),
  210--218.

\bibitem{korandi2018monochromatic}
{\sc Kor{\'a}ndi, D., Mousset, F., Nenadov, R., {\v{S}}kori{\'c}, N., and
  Sudakov, B.}
\newblock Monochromatic cycle covers in random graphs.
\newblock {\em Random Structures \& Algorithms 53}, 4 (2018), 667--691.

\bibitem{lng-scht-stn15}
{\sc Lang, R., Schaudt, O., and Stein, M.}
\newblock Almost partitioning a 3-edge-coloured $k_{n,n}$ into 5 monochromatic
  cycles.
\newblock {\em arXiv preprint arXiv:1510.00060\/} (2015).

\bibitem{lng-stn15}
{\sc Lang, R., and Stein, M.}
\newblock Local colourings and monochromatic partitions in complete bipartite
  graphs.
\newblock {\em Electronic Notes in Discrete Mathematics 49\/} (2015), 757--763.

\bibitem{letzter2019monochromatic}
{\sc Letzter, S.}
\newblock Monochromatic cycle partitions of $2 $-coloured graphs with minimum
  degree $3 n/4$.
\newblock {\em The Electronic Journal of Combinatorics 26}, 1 (2019), P1--19.

\bibitem{lczk-rdl-szmrd98}
{\sc {\L}uczak, T., R{\"o}dl, V., and Szemer{\'e}di, E.}
\newblock Partitioning two-coloured complete graphs into two monochromatic
  cycles.
\newblock {\em Combinatorics, Probability and Computing 7}, 04 (1998),
  423--436.

\bibitem{png-rdl-rcnsk02}
{\sc Peng, Y., R{\"o}dl, V., and Ruci{\'n}ski, A.}
\newblock Holes in graphs.
\newblock {\em Electron. J. Combin 9}, 1 (2002), 1--18.

\bibitem{pkvsk14}
{\sc Pokrovskiy, A.}
\newblock Partitioning edge-coloured complete graphs into monochromatic cycles
  and paths.
\newblock {\em Journal of Combinatorial Theory, Series B 106\/} (2014), 70--97.

\bibitem{sarkozy2011monochromatic}
{\sc S{\'a}rk{\"o}zy, G.~N.}
\newblock Monochromatic cycle partitions of edge-colored graphs.
\newblock {\em Journal of graph theory 66}, 1 (2011), 57--64.

\bibitem{schaudt2019partitioning}
{\sc Schaudt, O., and Stein, M.}
\newblock Partitioning two-coloured complete multipartite graphs into
  monochromatic paths and cycles.
\newblock {\em Journal of Graph Theory 91}, 2 (2019), 122--147.

\bibitem{stein2023monochromatic}
{\sc Stein, M.}
\newblock Monochromatic paths in 2-edge-coloured graphs and hypergraphs.
\newblock {\em Electron. J. Comb. 30}, 1 (2023).

\end{thebibliography}
\end{document}